\documentclass[12pt,leqno]{article}
\usepackage{amssymb,amsfonts,amsmath,amsthm,amscd,mathrsfs}
\usepackage{graphicx}
\def\IE{{\mathbb E}}
\def\IP{{\mathbb P}}
\def\IR{{\mathbb R}}

\def\IL{{\mathbb L}}

\def\IZ{{\mathbb Z}}

\def\n{\noindent}
\def\dsl{\textstyle\sum\limits}

\def\dis{\displaystyle}
\def\o{\omega}
\def\fr{\mbox{\footnotesize $\dis\frac{1}{2}$}}

\def\ov{\overline}
\def\ve{\varepsilon}
\def\f{\footnotesize}
\def\r{\rightarrow}

\def\wh{\widehat}
\def\wt{\widetilde}

\def\cA{{\cal A}}
\def\cB{{\cal B}}

\def\cD{{\cal D}}
\def\cE{{\cal E}}
\def\cL{{\cal L}}

\def\cE{{\cal E}}

\def\cM{{\cal M}}

\def\cF{{\cal F}}

\def\cV{{\cal V}}
\def\cW{{\cal W}}

\dimendef\dimen=0

\newtheorem{theorem}{Theorem}[section]
\newtheorem{lemma}[theorem]{Lemma}
\newtheorem{corollary}[theorem]{Corollary}
\newtheorem{proposition}[theorem]{Proposition}
\newtheorem{remark}[theorem]{Remark}

\begin{document}

\noindent

~

\bigskip
\begin{center}
{\bf LARGE DEVIATIONS FOR OCCUPATION TIME PROFILES  \linebreak OF RANDOM INTERLACEMENTS}
\end{center}

\begin{center}
Xinyi Li and Alain-Sol Sznitman
\end{center}

\bigskip\bigskip
\begin{center}
\end{center}

\begin{abstract}
We derive a large deviation principle for the density profile of occupation times of random interlacements at a fixed level in a large box of $\IZ^d$, $d \ge 3$. As an application, we analyze the asymptotic behavior of the probability that atypically high values of the density profile insulate a macroscopic body in a large box. As a step in this program, we obtain a similar large deviation principle for the occupation-time measure of Brownian interlacements at a fixed level in a large box of $\IR^d$, and we derive a new identity for the Laplace transform of the occupation-time measure, which is based on the analysis of certain Schr\"odinger semi-groups.
\end{abstract}

\vfill 

\n
Departement Mathematik 
ETH Z\"urich\\
CH-8092 Z\"urich\\
Switzerland

\vfill

\n
$\overline{~~~~~~~~}$
\\
{\footnotesize{This research was supported in part by the grant ERC-2009-AdG 245728-RWPERCRI}}
~
\newpage
\thispagestyle{empty}
~

\newpage
\setcounter{page}{1}

 \setcounter{section}{-1}
 
 \section{Introduction}
 \setcounter{equation}{0}
 
Random interlacements have been instrumental in the analysis of various questions concerning the disconnection or the fragmentation created by random walk trajectories, see \cite{CernTeixWind11},  \cite{Szni09b}, \cite{TeixWind11}. The existence of a non-trivial phase transition for the percolative properties of the vacant set of random interlacements, when one increases the level $u$ of the interlacement, plays an important role in their analysis. As it turns out, the level $u$ of random interlacements can also be measured by means of the random field of occupation times, which, in the case of $\IZ^d$, $d \ge 3$, is stationary, ergodic, and has average value $u$. In this work, we study the large deviations of the density profile of this random field in a large box of $\IZ^d$. As an application of the general large deviation principle we obtain, we analyze the asymptotic behavior of the probability that atypically high values of the density profile insulate a macroscopic body in a large box. One may naturally wonder whether such type of large deviations of the occupation-time profile actually captures the main mechanism for an atypical disconnection of a macroscopic body from infinity by the random interlacements, when the vacant set is in a percolative regime. In the course of our program, we derive a similar large deviation principle for the occupation-time measure of Brownian interlacements at a fixed level, in a large box of $\IR^d$, $d \ge 3$. The scaling invariance of Brownian interlacements permits to recast this problem in terms of general Cram\'er theory, and our results rely on a new identity for the Laplace transform of the occupation-time measure, which is based on the analysis of Schr\"odinger semi-groups.
 
 \medskip
 We now discuss our results in more detail. We consider continuous time random interlacements on $\IZ^d$, $d \ge 3$. In essence, this is a Poisson point process on a certain state space consisting of doubly infinite $\IZ^d$-valued trajectories marked by their duration at each step, modulo time-shift. A non-negative parameter $u$ comes as a multiplicative factor of the intensity measure of this Poisson point process, which is defined on a certain canonical space, see \cite{Szni11c}, which we denote here by $(\ov{\Omega}, \ov{\cA}, \ov{\IP})$. The field of occupation times of random interlacements at level $u$ is denoted by $L_{x,u}(\o)$, for $x \in \IZ^d$, $u \ge 0$, $\o \in \ov{\Omega}$. It records the total duration spent at $x$ by the trajectories modulo time-shift with label at most $u$ in the cloud $\o$, see \cite{Szni11c}.

\medskip
Given a {\it closed box} $B \subseteq \IR^d$ (by this we mean that $B$ is the product of $d$ non-degenerate compact intervals in $\IR$), a central object of interest in this work is the density profile of the occupation times at level $u$ in the large discrete box $(NB) \cap \IZ^d$: 
\begin{equation}\label{0.1}
\rho_{N,u} =  \dis\frac{1}{N^d} \;\dsl_{x \in (NB) \cap \IZ^d} L_{x,u} \;\delta_{\frac{x}{N}}, \;\; \mbox{for} \;\; N \ge 1, u \ge 0.
\end{equation}

\n
We view $\rho_{N,u}$ as a random element of $M_+(B)$, the set of positive measures on $B$, which we endow with the weak topology generated by $C(B)$, the set of continuous functions on $B$, and with its corresponding Borel $\sigma$-algebra. As a consequence of the ergodic theorem, see (\ref{4.4}), $\ov{\IP}$-a.s., $\rho_{N,u}$ converges weakly to $u \, m_B$, where $m_B$ stands for the restriction of the Lebesgue measure to Borel subsets of $B$. In Theorem \ref{theo5.8}, we establish a large deviation principle for $\rho_{N,u}$, which shows that for $u > 0$,
\begin{equation}\label{0.2}
\begin{array}{l}
\mbox{as $N \r \infty$, the laws of $\rho_{N,u}$ on $M_+(B)$ satisfy a large deviation principle}
\\
\mbox{at speed $N^{d-2}$ with the convex, good rate function $\frac{1}{d} \; I_u$},
\end{array}
\end{equation}

\n
where for $v >0$ and $\mu \in M_+(B)$, we have defined
\begin{equation}\label{0.3}
I_v(\mu)  =\left\{ \begin{split}
+ \infty, \;\;\; & \;\mbox{if $\mu$ is not absolutely continuous with respect to $m_B$},
\\
\inf \Big\{ \fr  & \dis\int_{\IR^d} |\nabla \varphi|^2 dy; \;\varphi \in H^1(\IR^d), \;\varphi = \mbox{\f $\sqrt{\dis\frac{d \mu}{dm_B}}$} - \sqrt{v}, \;\; \mbox{a.e.\,on} \; B\big\},
\\[0.5ex]
& \;\mbox{if $\mu$ is absolutely continuous with respect to $m_B$ (and the}
\\[-0.5ex]
&\; \mbox{infimum of the empty set equals $+ \infty$)}.
\end{split}\right.
\end{equation}

\medskip\n
In other words, $I_v$ is a non-negative, convex, lower semi-continuous function, with compact level sets $\{I_v \le a\}$, for $a \ge 0$, and for any open subset $O$ and closed subset $C$ of $M_+(B)$, we have

\vspace{-5ex}
\begin{equation}\label{0.4}
\begin{array}{lllll}
{\rm i)} &\quad  \limsup\limits_{N} & \dis\frac{1}{N^{d-2}} \; \log \ov{\IP} [\rho_{N,u} \in C] & \le & - \inf\limits_C \;\mbox{\f $\dis\frac{1}{d}$} \;I_u, \;\mbox{and} \qquad \qquad
\\[2ex]
{\rm ii)} &\quad  \liminf\limits_{N} & \dis\frac{1}{N^{d-2}} \; \log \ov{\IP} [\rho_{N,u} \in O] & \ge & - \inf\limits_O \;\mbox{\f $\dis\frac{1}{d}$} \;I_u\,.
\end{array}
\end{equation}
 
\medskip\n
 In the course of proving Theorem \ref{theo5.8}, we derive a large deviation principle for the density profile of the occupation-time measure of Brownian interlacements at level $\alpha > 0$ in a large box $LB$ of $\IR^d$, $d \ge 3$, which is of independent interest. Letting $\cL_\alpha$ stand for the random Radon measure of occupation times in $\IR^d$ of Brownian interlacements at level $\alpha$, see (\ref{1.23}) (or Section 2 of \cite{Szni13a}) we consider in place of $\rho_{N,u}$ the random measure $\nu_{L,\alpha}$ on $B$, such that for any bounded measurable $f$ on $B$, the integral $\langle \nu_{L,\alpha},f\rangle$ equals
 \begin{equation}\label{0.5}
 \big\langle \nu_{L,\alpha},f\big\rangle = \dis\frac{1}{L^d} \;\dis\int_{LB} f\Big(\mbox{\f $\dis\frac{y}{L}$}\Big) \, \cL_\alpha (dy), \; \mbox{for} \; L \ge 1, \alpha \ge 0.
 \end{equation}
 We show in Corollary \ref{cor3.3} that for $\alpha > 0$,
 \begin{equation}\label{0.6}
\begin{array}{lllll}
\mbox{as $L \r \infty$, the laws of $\nu_{L,\alpha}$ on $M_+(B)$ satisfy a large deviation principle},
\\
\mbox{at speed $L^{d-2}$ with the convex good rate function $I_\alpha$ (see (\ref{0.3})}).
\end{array}
\end{equation}

\n
As an application of the large deviation principle (\ref{0.2}), we analyze the asymptotic behavior of the probability that high values of the (smoothed-out) density profile insulate a large macroscopic body. We consider a regularization $f_N$ of $\rho_{N,u}$ obtained by the convolution of $\rho_{N,u}$ with a continuous probability density $\varphi_\delta$ supported in the closed Euclidean ball of radius $\delta$ centered at the origin of $\IR^d$. Given a compact subset $K$ of a sub-box $B_0$ of $B$ (at distance at least $\delta$ from $B^c$), and a positive number $a$, we are interested in the event $\{\rho_{N,u} \in \cD_{a,\delta}\}$ (see (\ref{6.5}) for the precise definition), where the level set $\{f_N \ge a\}$ ``disconnects $K$ from $\partial B_0$''. We show in Theorems \ref{theo6.2} and \ref{theo6.4} that when $a > u$,
\begin{equation}\label{0.7}
\begin{array}{lllll}
{\rm i)} &\quad  \limsup\limits_{N} & \dis\frac{1}{N^{d-2}} \; \log \ov{\IP} [\rho_{N,u} \in \cD_{a,\delta}] & \le & - \mbox{\f $\dis\frac{1}{d}$} \;\big(\sqrt{a} - \sqrt{u}\big)^2 {\rm cap}(K),
\\[2ex]
{\rm ii)} &\quad  \liminf\limits_{N} & \dis\frac{1}{N^{d-2}} \; \log \ov{\IP} [\rho_{N,u} \in \cD_{a,\delta}] & \ge & -\mbox{\f $\dis\frac{1}{d}$} \;\big(\sqrt{a} - \sqrt{u}\big)^2 {\rm cap}(K^\delta),
\end{array}
\end{equation}

\n
where ${\rm cap}(\cdot)$ stands for the Brownian capacity (see below (\ref{1.21})) and $K^\delta$ for the closed $\delta$-neighborhood of $K$. As $\delta \r 0$, one knows that ${\rm cap}(K^\delta)\downarrow {\rm cap}(K) $, see Remark \ref{rem6.5} 1), so that upper and lower bounds in (\ref{0.7}) become identical. Actually, one can let $\delta$ slowly tend to $0$ so that the corresponding upper and lower bounds match with the right-hand side of (\ref{0.7}) i), see Remark \ref{6.5} 2).

\medskip
The asymptotics (\ref{0.7}) has an interesting consequence. There is an intuitive strategy to ensure that the level set $\{f_N \ge a\}$ disconnects $K$ from $\partial B_0$ (or in other words that $\{\rho_{N,u} \in \cD_{a,\delta}\}$ occurs). Roughly speaking, it consists in inducing a suitable increase of the rate of the Poisson distribution of the number of bilateral trajectories with label at most $u$, which enter $(NK^\delta) \cap \IZ^d$. The lower bound (\ref{0.7}) ii) shows that this intuitive strategy is sub-optimal, see Remark \ref{rem6.5} 4). In essence, this strategy leads to a version of (\ref{0.7}) ii) where $(\sqrt{a} - \sqrt{u})^2$ is replaced by the strictly bigger quantity $a \log \frac{a}{u} - a + u$.

\medskip
Further, it is known that the vacant set $\cV^u$ of random interlacements at level $u$ undergoes a phase transition between a percolative phase when $u < u_*$, and a non-percolative phase when $u > u_*$, for a certain critical level $u_* \in (0,\infty)$ (see \cite{Szni10a}, \cite{SidoSzni09}, and also \cite{DrewRathSapo12}, \cite{PopoTeix13} for recent developments). When $u < u_*$, the vacant set of random interlacements is in a percolative regime. In the context of Bernoulli percolation, disconnecting a large macroscopic body in the percolative phase would involve an exponential cost proportional to $N^{d-1}$ (and surface tension), in the spirit of the study of the presence of a large finite cluster at the origin, see p.~216 of \cite{Grim99}, and Theorem 2.5, p.~16 of \cite{Cerf00}. In the present context, one may wonder whether large deviations of the density profile, as in (\ref{0.7}), with an exponential cost proportional to $N^{d-2}$,  capture the main mechanism ensuring that a macroscopic body gets disconnected from infinity by the interlacement at level $u$, when $u < u_*$. We refer to Remark \ref{rem6.5} 5) and to \cite{LiSzni13b} for more on this topic.

\medskip
It is also of interest to point out that the large deviation principle (\ref{0.2}) has some similar flavor to results of \cite{BoltDeus93} concerning the Gaussian free field (although the approaches in the two articles are quite different). This feature is in line with the isomorphism theorem of \cite{Szni12b}, which relates the field of occupation times of random interlacements to the square of Gaussian free fields.

\medskip
Let us give some comments concerning proofs. The large deviation principle (\ref{0.6}) (concerning the profile of the Brownian occupation-time measure) is used as a step in the proof of (\ref{0.2}). Due to the scaling property of $\cL_\alpha$, see (\ref{1.25}), $\nu_{L,\alpha}$ has the same distribution as $\wt{\nu}_{L,\alpha}$, the restriction to $B$ of $\frac{1}{L^{d-2}} \,\cL_{\alpha L^{d-2}}$. Since $\cL_\alpha$ has stationary and independent increments, (\ref{0.6}) can be proved by means of general Cram\'er theory and sub-additivity, see for instance \cite{DembZeit98}, p.~252. In this process, one important ingredient is a new identity for the Laplace functional of $\cL_\alpha$, which is based on methods of Schr\"odinger semi-groups, see \cite{Chen02}, \cite{ChenSong02}, \cite{ChunZhao95}. Indeed, we show in Theorem \ref{theo2.2} that for $\alpha \ge 0$, and any bounded measurable function $V$ on $\IR^d$, with compact support,
\begin{align}
&\IE [e^{\langle \cL_\alpha, V\rangle}] = e^{\alpha \Gamma(V)}, \; \mbox{where} \label{0.8}
\\[1ex]
&\Gamma(V) = \dis\int_{\IR^d} V \,dy + \sup\limits_{\varphi \in L^2 (\IR^d)} \Big\{2 \dis\int_{\IR^d} V \varphi \,dy + \dis\int_{\IR^d} V \varphi^2 dy - \cE(\varphi, \varphi)\Big\} \in (-\infty, + \infty], \label{0.9}
\end{align}
and for $\varphi \in L^2(\IR^d)$,
\begin{equation}\label{0.10}
\cE(\varphi,\varphi) =\left\{ \begin{array}{l}
\fr \;\dis\int_{\IR^d} |\nabla \varphi |^2 dy, \;  \mbox{if $\varphi \in H^1(\IR^d)$},
\\[2ex]
+ \infty, \;\mbox{otherwise},
\end{array}\right.
\end{equation}

\n
stands for the Dirichlet form attached to the Brownian semi-group (acting on $L^2(\IR^d)$), see also below (\ref{2.4}). Both members of (\ref{0.8}) may well be infinite. Remarkably, and unlike Proposition 2.6 of \cite{Szni13a} (see also Remark \ref{rem2.1} below), (\ref{0.8}) is an identity between extended numbers in $(0,+ \infty]$, which does not require any smallness assumption on $V$. We also refer to Remark \ref{rem2.3} for the discrete space counterpart of this identity.

\medskip
The proof of the main large deviation principle (\ref{0.2}) on the density profile $\rho_{N,u}$ appears in Theorem \ref{theo5.8}. It relies on sub-additivity and naturally splits into a lower bound and an upper bound. On the one hand, the lower bound (proved in Theorem \ref{theo5.4}) first relies on a lower bound stemming from sub-additivity. It then exploits a combination of the fact that the Brownian occupation-time measure is a scaling limit of occupation times of random interlacements on $\IZ^d$  (as proved in \cite{Szni13a} and recalled in (\ref{4.5}) below), and the key large deviation result on $\frac{1}{\alpha} \,\cL_{\alpha}$ restricted to a box (proved in Section 3). On the other hand, the upper bound (proved in Theorem \ref{theo5.5}) combines an upper bound stemming from sub-additivity, involving a discrete version of the functional in (\ref{0.3}), and an estimate in the spirit of Gamma-convergence (see Chapter 7 of \cite{Dalm93}), which compares the large $N$ behavior of a sequence of variational problems on the scaled lattices $\frac{1}{N}\,\IZ^d$ to a suitable continuous-space variational problem.

\medskip
The asymptotic bounds (\ref{0.7}) on the probability of insulation of a large macroscopic body by high values of the (smoothed-out) occupation-time profile (see Theorems \ref{theo6.2} and \ref{theo6.4}) are direct applications of the main large deviation principle (\ref{0.2}) and of the structure of the rate function, see Lemma \ref{3.1}.

\medskip
Let us now describe the organization of this article. Section 1 introduces further notation, collects material concerning Schr\"odinger semi-groups, and recalls some properties of Brownian interlacements. The main objective of Section 2 is to establish the identity (\ref{2.5}) in Theorem \ref{theo2.2}. Some of the key consequences of this identity appear in the Corollaries \ref{cor2.4} and \ref{cor2.5}. The discrete space situation is discussed in Remark \ref{rem2.3}. Section 3 derives in Theorem \ref{theo3.2} a large deviation principle for the Brownian interlacement case, which plays a central role. Its application to the proof of (\ref{0.6}) appears in Corollary \ref{cor3.3}. The main properties of the rate function $I_v$ are collected in Lemma \ref{lem3.1}. The short Section 4 describes the (scaled) discrete space set-up and some of the results following by the methods of Section 2 in this context, see Theorem \ref{theo4.1} and Corollary \ref{cor4.2}. Section 5 is devoted to the proof of (\ref{0.2}), see Theorem \ref{theo5.8}. The lower bound appears in Theorem \ref{theo5.4} and the upper bound in Theorem \ref{theo5.5}. The main sub-additive estimates are contained in Corollary \ref{cor5.3}, and the relevant form of the scaling limit of occupation times in Lemma \ref{lem5.1}. The last Section 6 contains the proof of (\ref{0.7}), see Theorems \ref{theo6.2} and \ref{theo6.4}. Extensions are discussed at the end, in Remark \ref{rem6.5}.

\medskip
Finally, let us explain the convention we use concerning constants. We denote by $c,c',\ov{c}$ positive constants changing from place to place, which simply depend on $d$. Numbered constants $c_0,c_1,\dots$ refer to the value corresponding to their first appearance in the text. Dependence of constants on additional parameters appears in the notation.

\pagebreak
\section{Some useful facts about Schr\"odinger semi-groups and Brownian interlacements}
\setcounter{equation}{0}

In this section, we first introduce some further notation. We collect some rather classical properties of Schr\"odinger semi-groups and gauge functions, which will be useful in the next section, see Proposition \ref{prop1.2} and Lemma \ref{lem1.3}. Moreover, we recall some properties of Brownian interlacements on $\IR^d$. Further facts concerning  continuous time random interlacements on $\IZ^d$ will appear in Section 4. Throughout, we tacitly assume that $d \ge 3$.

\medskip
We write $| \cdot |$ and $| \cdot |_\infty$ for the Euclidean and the supremum norms on $\IR^d$. We denote by $B(y,r)$ the closed Euclidean ball with center $y \in \IR^d$ and radius $r \ge 0$. We write $B_\infty(y,r)$ in the case of the supremum norm. Given $A,B \subseteq \IR^d$, we denote by $d(A,B) = \inf\{|y - y'|$; $y \in A$, $y' \in B\}$ the mutual Euclidean distance of $A$ an $B$. When $A = \{y\}$, we simply write $d(y,B)$ in place of $d(\{y\}, B)$. We define $d_\infty(A,B)$ and $d_\infty(y,B)$ analogously, with $| \cdot |_\infty$ in place of $| \cdot |$. The shorthand notation $K \subset\subset \IR^d$, resp.~$K \subset\subset \IZ^d$, means that $K$ is compact subset of $\IR^d$, resp.~a finite subset of $\IZ^d$. We denote by $f_+ = \max\{f,0\}$, $f_- = \max\{-f,0\}$, the positive and negative part of a function $f$. We routinely write $\langle \nu,f\rangle$ to denote the integral with respect to a measure $\nu$ of a measurable, non-negative, or $\nu$-integrable, function $f$. When $f,h$ are measurable functions on $\IR^d$ such that $|f h|$ is Lebesgue-integrable, we write $\langle f,h \rangle = \int_{\IR^d} f(y) h(y) dy$. We denote by $\|f\|_\infty$ the supremum norm of the function $f$, and by $\|f\|_p$ its $L^p$-norm $(1\le p < \infty)$. We specify the relevant $L^p$-space in the notation when there might be some ambiguity, and write for instance $\|f\|_{L^p(\IR^d,dy)}$ or $\|f\|_{L^p(B,\nu)}$.

\medskip
We denote by $W_+$ the subspace of $C(\IR_+, \IR^d)$ of continuous $\IR^d$-valued trajectories tending to infinity at infinite times. We write $(X_t)_{t \ge 0}$ for the canonical process, and denote by $(\theta_t)_{t \ge 0}$ the canonical shift. We endow $W_+$ with the $\sigma$-algebra $\cW_+$ generated by the canonical process. Given an open set $U$ of $R^d$, $w \in W_+$, we write $T_U(w) = \inf\{s \ge 0; X_s(w) \notin U\}$, for the exit time of $U$. When $F$ is a closed subset of $\IR^d$, we write $H_F(w) = \inf\{s \ge 0$; $X_s(w) \in F\}$, and $\wt{H}_F (w) = \inf\{s > 0; X_s(w) \in F\}$, for the respective entrance, and hitting times of $F$. We assume $d \ge 3$, so that Brownian motion on $\IR^d$ is transient, and we view $P_y$, the Wiener measure starting from $y \in \IR^d$, as defined on $(W_+, \cW_+)$. We denote by $E_y$ the corresponding expectation. When $\rho$ is a finite measure on $\IR^d$, we write $P_\rho$ for the Wiener measure with ``initial distribution'' $\rho$ and $E_\rho$ for the corresponding expectation. We write $p(t,y,y') = (2 \pi t)^{-\frac{d}{2}} \exp\{-\frac{|y-y'|^2}{2 t}\}$, with $t > 0$, $y,y' \in \IR^d$, for the Brownian transition density.

\medskip
We now recall some properties of the Schr\"odinger semi-groups we consider here. We denote by $L^\infty_c(\IR^d)$ the space of bounded measurable functions $V$ on $\IR^d$, which vanish outside a compact subset of $\IR^d$. Given $V \in L^\infty_c(\IR^d)$, we introduce the Schr\"odinger semi-group attached to $V$, namely  the strongly continuous self-adjoint semi-group on $L^2(\IR^d,dy)$, see Proposition 3.3, p.~16 of \cite{Szni98a},
\begin{align}
R^V_t  f(y) & = E_y\big[f(X_t)\,e^{\int^t_0 V(X_s)ds}\big], \; \mbox{for} \; f \in L^2(\IR^d), t \ge 0, y \in \IR^d, \label{1.1}
\\
& = \mbox{\f $\dis\int$}_{\IR^d} r_V(t,y,y') \,f(y')dy', \; \mbox{when $t > 0$, where} \nonumber
\\
r_V(t,y,y') & = p(t,y,y')\,E^t_{y,y'} \Big[\exp\Big\{\mbox{\f $\dis\int$}^t_0 V(X_s)ds\Big\}\Big], \; \mbox{for $t >0, y,y' \in \IR^d$},\label{1.2}
\end{align}

\pagebreak\n
is a symmetric function of $y,y'$, see Proposition 3.1, p.~13-14 of \cite{Szni98a}, which is jointly continuous, see Proposition 3.5, p.~18 of \cite{Szni98a}, and $E^t_{y,y'}$ stands for the expectation corresponding to $P^t_{y,y'}$, the Brownian bridge measure in time $t$ from $y$ to $y'$, see p.~137-140 of  \cite{Szni98a}. As an immediate consequence of (\ref{1.2}),
\begin{equation}\label{1.3}
e^{-\|V\|_\infty t} p(t,y,y') \le r_V(t,y,y') \le e^{\|V\|_\infty t}\, p(t,y,y'), \; \mbox{for} \; t > 0, y,y' \in \IR^d.
\end{equation}

\n
We now turn to the discussion of the Green operators corresponding to the Schr\"odinger semi-groups. We thus consider $V \in L^\infty_c(\IR^d)$ as above and define
\begin{equation}\label{1.4}
G_V  f(y) = E_y\Big[ \dis\int^\infty_0 f(X_s) \,e^{\int^s_0 V(X_u)du}ds\Big] =  \dis\int_{\IR^d}\dis\int^\infty_0 r_V(s,y,y')\,f(y')ds\,dy',
\end{equation}

\medskip\n
for $f$ a measurable non-negative function on $\IR^d$, and $y \in \IR^d$.

\medskip
When $V = 0$ (so $r_V(t,y,y') = p(t,y,y')$) we simply write $G$, and recover the usual Green operator attached to Brownian motion
\begin{equation}\label{1.5}
G f(y) = E_y \Big[ \dis\int^\infty_0 f(X_s)ds\Big] = \mbox{\f $\dis\frac{\Gamma(\frac{d}{2} -1)}{2 \pi^{\frac{d}{2}}}$} \;\dis\int_{\IR^d} \; \mbox{\f $\dis\frac{1}{|y - y'|^{d-2}}$} \;f(y')dy.
\end{equation}

\n
We introduce in the (classical) lemma below  a condition corresponding to the so-called sub-criticality (of $\frac{1}{2} \,\Delta + V$) see p.~145, as well as pp.~129, 136, 150 of \cite{Pins95}. Our assumptions are slightly different and we briefly sketch the proof for the reader's convenience.

\begin{lemma}\label{lem1.1} (recall $V \in L^\infty_c(\IR^d)$)

\medskip
Assume that
\begin{equation}\label{1.6}
\begin{array}{l}
G_V W(y_0) < \infty, \; \mbox{for some $y_0 \in \IR^d$ and some $[0,1]$-valued, measurable}\\
\mbox{function $W$ not a.e.\,equal to $0$}.
\end{array}
\end{equation}
Then, for any bounded open set $U$,
\begin{equation}\label{1.7}
\mbox{$G_V \,1_U$ is a bounded function}.
\end{equation}
Moreover, in the notation of (\ref{1.1}), the Schr\"odinger semi-group
\begin{equation}\label{1.8}
\begin{array}{l}
\mbox{$(R^V_t)_{t \ge 0}$ is a strongly continuous semi-group of self-adjoint contractions }\\
\mbox{on $L^2(\IR^d)$}.
\end{array}
\end{equation}
\end{lemma}

\begin{proof}
We first sketch the proof of (\ref{1.7}). We note that
\begin{equation*}
r_V(1,y_0,y) \ge c(V,y_0,y_1)\,r_V\Big(\fr, y_1,y\Big), \; \mbox{for} \; y_0,y_1, y \in \IR^d,
\end{equation*}

\n
with $c(V,\cdot,\cdot)$ locally bounded away from zero. Indeed, this follows from (\ref{1.3}) and the inequality $p(1,0,z)/p(\frac{1}{2},z_1,z) = 2^{-\frac{d}{2}} \exp\{\frac{1}{2} \,|z-2z_1|^2 - |z_1|^2\} \ge 2^{-\frac{d}{2}} \exp\{-|z_1|^2\}$, for $z,z_1 \in \IR^d$, combined with translation invariance (set $z =y - y_0$, $z_1 = y_1 - y_0$). Thus, for $y_1 \in \IR^d$, we find that the above inequality together with the semi-group property yields that
\begin{equation*}
\begin{array}{l}
\infty > G_V W(y_0) \stackrel{(\ref{1.4})}{\ge} \dis\int^\infty_0 \dis\int r_V(s+1,y_0,y) \,W(y)dy\,ds \ge \
\\
c(V,y_0,y_1) \dis\int^\infty_{\frac{1}{2}} \dis\int r_V(t,y_1,y)\,W(y)dy\,dt,
\end{array}
\end{equation*}

\n
which combined with (\ref{1.3}) (for the values $t \in (0,\frac{1}{2}])$ implies that $G_V W$ is locally bounded. In addition, by the semi-group property, we see that for $y \in \IR^d$,
\begin{equation*}
G_V W(y) \ge \dis\int^\infty_0 \dis\int r_V(s+1,y,y') \,W(y')dy'ds \ge G_V 1_U(y) \inf\limits_{z \in U} R^V_1 W.
\end{equation*}

\n
By our assumption on $W$, the last term is positive, and hence, $G_V 1_U$ is locally bounded. Choosing $K \subset \subset \IR^d$ containing $U$ and the support of $V$, the strong Markov property yields,
\begin{equation*}
G_V 1_U(y) = E_y\big[(G_V 1_U)(X_{H_K}), H_K < \infty\big], \; \mbox{for any} \; y \in \IR^d,
\end{equation*}
so $G_V 1_U$ is bounded and (\ref{1.7}) follows.

\medskip
We now turn to the proof of (\ref{1.8}). For $\varphi \in L^\infty_c(\IR^d)$ we denote by $d E_{\varphi,\varphi}(\lambda)$ the spectral measure of $\varphi$ (and $E$ is a spectral resolution of the identity of the generator of $(R^V_t)_{t \ge 0}$, see for instance Theorems 13.30 and 13.37, pp.~348, 360 of \cite{Rudi74a}).We find by (\ref{1.7}) that 
\begin{equation*}
\infty > \dis\int^\infty_0 dt \big\langle R^V_t  \varphi, \varphi \big\rangle = \dis\int^\infty_0 dt \dis\int_\IR e^{-\lambda t} dE_{\varphi,\varphi}(\lambda),
\end{equation*}

\n
 Hence, $dE_{\varphi,\varphi}$ gives no mass to $(-\infty,0]$, and therefore for $\varphi \in L^\infty_c(\IR^d)$,
\begin{equation*}
\big\langle R^V_t \varphi, R^V_t \varphi \big\rangle = \dis\int^\infty_0 e^{-2\lambda t} dE_{\varphi,\varphi}(\lambda) \le \dis\int^\infty_0 dE_{\varphi,\varphi}(\lambda) = \|\varphi\|^2_{L^2(\IR^d)}.
\end{equation*}
Since $L^\infty_c(\IR^d)$ is dense in $L^2(\IR^d)$, the claim (\ref{1.8}) follows.
\end{proof}

We now recall some properties of the gauge function
\begin{equation}\label{1.9}
\gamma_V(y) = E_y \big[e^{\int^\infty_0 V(X_s)ds}\big], \; \mbox{for} \; y \in \IR^d
\end{equation}

\n
(note that $V \in L^\infty_c(\IR^d)$ and the integral is finite due to transience, or more precisely, to the fact that $P_y$ is a probability on $W_+$, see above (\ref{1.1})). As the next proposition shows, the gauge function is closely related to the Schr\"odinger semi-group (attached to $V$) via its Green operator. We refer to \cite{Chen02}, \cite{ChenSong02}, and \cite{ChunZhao95}, for much more on the subject.

\begin{proposition}\label{prop1.2} (recall $V \in L^\infty_c (\IR^d)$)

\medskip
The condition (\ref{1.6}) is equivalent to
\begin{equation}\label{1.10}
\mbox{$\gamma_V$ is not identically infinite (Gauge Condition)}.
\end{equation}

\medskip\n
If (\ref{1.6}), or equivalently (\ref{1.10}), holds, then
\begin{align}
&\mbox{$\gamma_V$ is a bounded continuous function on $\IR^d$ tending to $1$ at infinity}, \label{1.11}
\\[1ex]
& \gamma_V = 1 + GV\gamma_V = 1 + G_V V. \label{1.12}
\end{align}
\end{proposition}

\begin{proof}
To show that (\ref{1.6}) implies (\ref{1.10}), we use the identity
\begin{equation}\label{1.13}
e^{\int^t_0 V(X_s)ds} = 1 + \dis\int^t_0 V(X_s) \,e^{\int^s_0 V(X_u) du} ds, \; \mbox{for $t \ge 0$}.
\end{equation}

\medskip\n
By Fatou's lemma, (\ref{1.13}) implies that $\gamma_V(y) \le 1 + G_V |V|(y)$, which is a bounded function of $y$, by (\ref{1.7}). The fact that $\gamma_V$ coincides with the last expression of (\ref{1.12}) follows by dominated convergence.

\medskip
To prove that (\ref{1.10}) implies (\ref{1.6}), either $V_+$ ($= \max (V,0)$) vanishes a.e., so that $r_V \le p$ and (\ref{1.6}) holds, or else, by Theorem 2.8, p.~4651 of \cite{Chen02}, $\|G_V V_+\|_\infty < \infty$ and $V_+$ is not a.e.\,equal to $0$. This implies (\ref{1.6}) (choosing $W = V_+ \wedge 1$). Thus (\ref{1.6}) and (\ref{1.10}) are equivalent.

\medskip
To prove (\ref{1.11}), we already know from the discussion below (\ref{1.13}) that $\gamma_V$ is a bounded function. It is continuous by the Corollary, p.~150 of \cite{ChunZhao95},  or (ii) in Theorem 4.7, p.~115 of the same reference. The fact that $\gamma_V$ tends to $1$ at infinity follows from the first equality of (\ref{1.12}), which we prove now. To derive the first equality of (\ref{1.12}) (and this will compete the proof of Proposition \ref{prop1.2}), we use the identity
\begin{equation}\label{1.14}
e^{\int^t_0 V(X_s)ds} = 1 + \dis\int^t_0 V(X_s) \,e^{\int^t_s V(X_u) du} ds, \; \mbox{for $t \ge 0$}.
\end{equation}

\n
Integrating with respect to $P_y$ and using the Markov property yields 
\begin{equation*}
E_y\big[e^{\int^t_0 V(X_s)ds}\big] = 1 + E_y\Big[\dis\int^t_0 V(X_s) \, E_{X_s} \big[e^{\int^{t-s}_0 V(X_u)du}\big] ds\Big], \;\mbox{for $y \in \IR^d$}.
\end{equation*} 

\n
By (\ref{1.7}) and (\ref{1.13}) we see that the inner expectation is uniformly bounded, and converges to $\gamma_V(X_s)$ as $t \r \infty$. The first equality of (\ref{1.12}) now follows by dominated convergence. The second equality results from the discussion below (\ref{1.13}), and the proof of Proposition \ref{prop1.2} is now complete.
\end{proof}

The following approximation lemma will be helpful in Sections 2 and 3. We say that $r$ and $r'$ in $(1,\infty)$ are conjugate exponents when $\frac{1}{r} + \frac{1}{r'} = 1$.

\begin{lemma}\label{lem1.3}
Consider $r > \frac{d}{2}$, and $V, V_n, n \ge 1$, in $L^\infty_c(\IR^d)$, which all vanish outside $K \subset \subset \IR^d$. Assume that (\ref{1.6}) holds for $V$ and that
\begin{equation}\label{1.15}
\lim\limits_n \,\|V - V_n\|_{L^r(\IR^d)} = 0.
\end{equation}
Then, for large $n$, (\ref{1.6}) holds for $V_n$, and
\begin{equation}\label{1.16}
\mbox{$\gamma_{V_n}$ converges uniformly to $\gamma_V$ on $\IR^d$}.
\end{equation}
\end{lemma}

\begin{proof}
By Theorem 2.17, p.~4660 of \cite{Chen02} (Super Gauge Theorem), for some $p > 1$, $pV$ satisfies (\ref{1.10}) and hence (\ref{1.6}). Thus, denoting by $q$ the conjugate exponent of $p$,
\begin{equation}\label{1.17}
\begin{array}{lcl}
|\gamma_{V_n}(y) - \gamma_V(y)| &\!\!\!\! \le & \!\!\!\!  E_y \big[\big|e^{\int^\infty_0(V_n-V)(X_s)ds} - 1\big|e^{\int^\infty_0 V(X_s)ds}\big] 
\\[1ex]
&\!\!\!\! \stackrel{\mbox{\small H\"older}}{\le}  &\!\!\!\!  E_y \big[\big|e^{\int^\infty_0  (V_n-V)(X_s)ds} - 1\big|^q\big]^{\frac{1}{q}} \,\big\|\gamma_{pV}\big\|^{\frac{1}{p}}_\infty, \; \mbox{for all} \; y \in \IR^d
\end{array}
\end{equation}
(and $\| \gamma_{pV}\|_\infty$ is finite by (\ref{1.11})).

\medskip
Thus, choosing $2m \ge q$, with $m \ge 1$ integer, and setting $\Delta_n = V_n - V$, we find that the $(2m)$-th power of the first term in the last line of (\ref{1.17}) is smaller than
\begin{equation*}
E_y \big[ \big(e^{\int^\infty_0\Delta_n(X_s)ds} - 1\big)^{2m}\big] = \dsl^{2m}_{\ell = 0} \; {2m \choose \ell} (-1)^{2m - \ell} E_y \big[e^{\int^\infty_0\ell \Delta_n(X_s)ds} \big].
\end{equation*}
The claim (\ref{1.16}) will thus follow once we show that for each fixed $\ell \ge 0$,
\begin{equation}\label{1.18}
\sup\limits_{y \in \IR^d} \big| E_y\big[e^{\ell\int^\infty_0 \Delta_n(X_s)ds} \big] - 1 \big| \underset{n}{\longrightarrow} 0\,.
\end{equation}

\n
Expanding the exponential and using the Markov-property, we can bound the absolute value in (\ref{1.18}) by $\sum_{k \ge 1} \ell^k \|G \,|\Delta_n|\, \|^k_\infty$.

\medskip
By assumption, $r > \frac{d}{2}$, so, the conjugate exponent $r'$ belongs to $(1, \frac{d}{d-2})$. Then, for $z \in \IR^d$, 
\begin{align*}
G|\Delta_n|(z) \stackrel{(\ref{1.5})}{=} c \dis\int \dis\frac{|\Delta_n(y)|}{|y-z|^{d-2}} \;dy & \stackrel{\mbox{\f H\"older}}{\le} c \Big(\dis\int_K\dis\frac{1}{|y-z|^{(d-2)r'}}\,dy\Big)^{\frac{1}{r'}} \|\Delta_n\|_r
\\[1ex]
&\;\;\;  \le c(r,K) \,\|\Delta_n\|_r 
\end{align*}

\n
(for the last bound, where the dependence in $z$ has disappeared, one considers the smallest $R \ge 1$ such that $B(0,R) \supseteq K$, and looks separately at $z \notin B(0,2R)$, or $z \in B(0,2R)$ and hence $K \subseteq B(z,3R)$). So, we see that for $\ell \le 2 m$,
\begin{equation}\label{1.19}
\begin{split}
\sup\limits_{y \in \IR^d} \big|E_y\big[e^{\ell \int^\infty_0 \Delta_n(X_s)ds}\big] - 1\big| & \le \dsl_{k \ge 1} (2m)^k \|G\,|\Delta_n|\,\|_\infty^k
\\
& \le \dis\frac{c(r,K)\,2m\|\Delta_n\|_r}{(1-c(r,K)\,2m \|\Delta_n\|_r)_+} \; \underset{n}{\stackrel{(\ref{1.5})}{\longrightarrow}} 0.
\end{split}
\end{equation}

\n
This proves (\ref{1.18}) and (\ref{1.16}) follows.
\end{proof}

We now recall some properties of Brownian interlacements, and refer to Section 2 of \cite{Szni13a} for more detail. Brownian interlacements correspond to a certain Poisson point process on a state space, which is the product $W^* \times \IR_+$, where $W^*$ stands for the space of continuous trajectories from $\IR$ into $\IR^d$, tending to infinity at plus and minus infinite times, modulo time-shift. The intensity measure of this Poisson point process is the product of a certain $\sigma$-finite measure $\nu$ on $W^*$ (see Theorem 2.2 of \cite{Szni13a}), with the Lebesgue measure $d \alpha$ on $\IR_+$. Informally, this point process corresponds to a cloud of doubly-infinite trajectories modulo time-shift having each a non-negative label (the $\IR_+$-component of $W^* \times \IR_+$). The Poisson point process is defined on a certain canonical space $(\Omega, \cA, \IP)$, see (2.23) of \cite{Szni13a}. We collect below some properties of Brownian interlacements, which we will use here.

\medskip
Given $K \subset \subset \IR^d$, $\alpha \ge 0$, and $\o \in \Omega$, one considers the point measure on $W_+$, denoted by $\mu_{K,\alpha}(\o)$, which collects for all bilateral trajectories with label at most $\alpha$, which enter $K$ at some point, their forward trajectories after their first entrance time in $K$. Then, see (2.25) of \cite{Szni13a},
\begin{align}
&\mbox{$\mu_{K,\alpha}$ is a Poisson point process on $W_+$ with intensity measure $\alpha P_{e_K}$, where}\label{1.20}
\\[1ex]
& \mbox{$e_K(dy)$ stands for the equilibrium measure of $K$}. \label{1.21}
\end{align}

\medskip\n
The equilibrium measure of $K$ is a finite measure concentrated on $\partial K$, and its total mass is called the capacity of $K$, see pp.~58, 61 of \cite{Szni98a}. Moreover,
\begin{equation}\label{1.22}
\langle e_K, Gf\rangle = \langle 1,f \rangle \;\; \mbox{if $f \in L^\infty_c(\IR^d)$ vanishes outside $K$}.
\end{equation}

\n
One also introduces the occupation-time measure $\cL_\alpha(\o)$ of Brownian interlacements at level $\alpha \ge 0$ in the cloud $\omega$. It is the Radon measure on $\IR^d$, which to each $A \in \cB(\IR^d)$ gives a mass equal to the total time spent in $A$ by all trajectories modulo time-shift with label at most $\alpha$ in the cloud $\omega$. In particular, when $V \in L^\infty_c(\IR^d)$ vanishes outside $K \subset \subset \IR^d$, one has for $\alpha \ge 0$, and $\o \in \Omega$,
\begin{equation}\label{1.23}
\big\langle \cL_\alpha(\o), V \big\rangle = \big\langle \mu_{K,\alpha}(\o), f_V\big\rangle, \; \mbox{where} \; f_V(w) = \dis\int^\infty_0 V\big(X_s(w)\big)ds, \; \mbox{for} \;w \in W_+ .
\end{equation}

\n
The intensity measure of $\cL_\alpha$ equals $\alpha \,dy$, cf.~(2.38) of \cite{Szni13a}, that is
\begin{equation}\label{1.24}
\IE\big[ \big\langle \cL_\alpha, V \big\rangle\big] = \alpha \dis\int_{\IR^d} V \,dy, \; \mbox{for all} \; V \in L^\infty_c(\IR^d).
\end{equation}

\n
Moreover, $\cL_\alpha$ has an important scaling property, see (2.43) of \cite{Szni13a},
\begin{equation}\label{1.25}
\cL_{\lambda^{2-d} \alpha} \stackrel{\rm law}{=} \lambda^2 h_\lambda \circ \cL_\alpha, \; \mbox{for} \; \lambda > 0,
\end{equation}

\n
where $h_\lambda \circ \cL_\alpha$ stands for the image of $\cL_\alpha$ under the homothety of ratio $\lambda$ on $\IR^d$.

\medskip
One also has an expression for the Laplace transform of $\cL_\alpha$ in the ``neighborhood of the origin'' (see Proposition 2.6 of \cite{Szni13a}). Namely for $V \in L^\infty_c(\IR^d)$ such that $\|G\, |V|\,\|_\infty < 1$, one has
\begin{equation}\label{1.26}
\IE[e^{\langle \cL_\alpha, V\rangle}] =  \exp\{\alpha \langle V, (I-GV)^{-1} 1\rangle\}
\end{equation}

\n
(the assumption on $V$ ensures that $I-GV$ operating on $L^\infty(\IR^d)$ has a bounded inverse). In the next section we will derive identities that bypass the smallness assumption on $V$ in (\ref{1.26}), and remain true even when the left-hand side of (\ref{1.26}) explodes.

\medskip
We close this section with a lemma about Poisson point processes, which will be helpful in the next section. We consider a measurable space $(E, \cE)$.

\begin{lemma}\label{lem1.4}
Let $\mu$ be a Poisson point process on $E$ with finite intensity measure $\eta$ (i.e. $\eta(E) < \infty$), and let $\Phi$: $E \r \IR$ be a measurable function. Then, one has
\begin{equation}\label{1.27}
\IE[e^{\langle \mu, \Phi \rangle}] = \exp\Big\{\dis\int_E(e^\Phi - 1)d\eta\Big\}
\end{equation}

\n
(this is an identity between numbers in $(0,+ \infty]$).
\end{lemma}

\begin{proof}
Set $\Phi_n = \sum_{\IZ \ni k < n2^n} \; \frac{k}{2^n} \;1\{ \frac{k}{2^n} \le \Phi < \frac{k+1}{2^n}\}$, for $n \ge 1$, so that $\Phi_n$ is measurable, $(-\infty,n]$-valued, and $\Phi_n$ increases to $\Phi$, as $n \r \infty$. Then, for each $n$, classically, see for instance \cite{Resn87}, p.~130-134,
\begin{equation*}
\IE[e^{\langle \mu, \Phi_n\rangle}] = \exp\Big\{\dis\int_E(e^{\Phi_n}-1)d\eta\Big\} \;\mbox{($< \infty$, since $\Phi_n \le n$ and $\eta(E) < \infty$)},
\end{equation*}

\n
and (\ref{1.27}) follows by monotone convergence.
\end{proof}

\section{Laplace functional of occupation-time measures of Brownian interlacements}
\setcounter{equation}{0}

In this section, with the help of Schr\"odinger semi-groups techniques (in particular Lemma \ref{lem1.1} and Proposition \ref{prop1.2} of the previous section), we derive an identity for the Laplace functional of $\cL_\alpha$, see Theorem \ref{theo2.2}. This identity plays an important role for the identification of the rate function governing the large deviation principle for the occupation-time profile of Brownian interlacements in a large box, which we derive in the next section. We state two consequences of the basic identity (\ref{2.5}), see Corollaries \ref{cor2.4} and \ref{cor2.5}. In Remark \ref{rem2.3}, we discuss the corresponding identity one obtains in the case of continuous-time random interlacements on a transient weighted graph.

\medskip
We begin with the observation that the Laplace functional of  $\cL_\alpha$ naturally involves the gauge function. We consider $V\in L^\infty_c(\IR^d)$ and recall that $\gamma_V$ stands for the gauge function, see (\ref{1.9}). When $V$ vanishes on $K^c$, with $K \subset \subset \IR^d$, as in (\ref{1.23}), we  can express $\langle \cL_\alpha, V\rangle$ in terms of the Poisson point process $ \mu_{K,\alpha}$ introduced above (\ref{1.20}), and use Lemma \ref{lem1.4} to compute the exponential moment of $\langle \cL_\alpha, V\rangle$. We find
\begin{equation}\label{2.1}
\begin{split}
\IE[e^{\langle \cL_\alpha, V\rangle}] & \stackrel{(\ref{1.23})}{=} \IE[e^{\langle \mu_{K,\alpha}, f_V\rangle}] \stackrel{(\ref{1.20}), ( \ref{1.27})}{=} e^{\alpha \langle e_K,E_x[e^{\int^\infty_0 V(X_s) ds}] - 1\rangle}
\\[1ex]
& \stackrel{(\ref{1.9})}{=} e^{\alpha \langle e_K, \gamma_V - 1\rangle}, \; \mbox{for $\alpha \ge 0$}
\end{split}
\end{equation}
(note that this is an identity between numbers in $(0,\infty]$, and there is no smallness assumption on $V$).

\begin{remark}\label{rem2.1} \rm It is well-known that the gauge function $\gamma_V$ (and hence the left-hand side of (\ref{2.1})) may very well be infinite, see for instance \cite{Pins95}, p.~227, and pp.~166, 167. This feature complicates the study of the Laplace transform of $\cL_\alpha$. When (\ref{1.6}) (or equivalently (\ref{1.10})) holds, we know from (\ref{1.11}), (\ref{1.12}) that $\gamma_V$ is bounded continuous and $\gamma_V - 1 = GV \gamma_V$. By (\ref{1.22}) we thus find that
\begin{equation}\label{2.2}
\IE[e^{\langle \cL_\alpha,V\rangle}] = e^{\alpha \langle V,\gamma_V\rangle}, \; \mbox{for $\alpha \ge 0$, when (\ref{1.6}) holds}. 
\end{equation}

\n
In particular, when $\|G |V|\,\|_\infty < 1$, then (\ref{1.10}) holds (see for instance (2.6) in \cite{Szni11c}), and $I-GV$ operating on $L^\infty(\IR^d)$ has a bounded inverse. So, by the first equality in (\ref{1.12})$, \gamma_V = (I-GV)^{-1} 1$, and we find $\langle V, \gamma_V\rangle = \langle V, (I-GV)^{-1} 1\rangle$. In this fashion we recover (\ref{1.26}) out of (\ref{2.2}) when $\|G |V|\,\|_\infty < 1$. \hfill $\square$
\end{remark}

We introduce a notation for the logarithm of the Laplace functional of $\cL_1$. For $V \in L^\infty_c(\IR^d)$, we set
\begin{equation}\label{2.3}
\Lambda(V) =  \mbox{\f $\dis\frac{1}{\alpha}$} \;\log \IE[e^{\langle \cL_\alpha,V\rangle}] \in (-\infty, + \infty] \; \mbox{(and $\alpha > 0$ is arbitrary by (\ref{2.1})}).
\end{equation}
We also recall from (\ref{0.9}) the notation
\begin{equation}\label{2.4}
\Gamma(V) = \dis\int_{\IR^d} V \,dy + \sup\limits_{\varphi \in L^2(\IR^d)} \{ 2 \langle V, \varphi \rangle + \langle V \varphi, \varphi \rangle - \cE(\varphi,\varphi)\},
\end{equation}

\n
where $\cE(\varphi,\varphi)$ denotes as in (\ref{0.10}) the Dirichlet form attached to the Brownian semi-group acting on $L^2(\IR^d)$ (i.e. $(R_t^{V=0})_{t \ge 0}$ in the notation of (\ref{1.11})), see for instance \cite{Szni98a}, p.~26. Further, we will also use the notation $\cE(\varphi,\varphi)$, when $\varphi$ belongs to the extended Dirichlet space $\cF_e$ (consisting of functions that are a.e.\,limits of sequences in $H^1(\IR^d)$ that are Cauchy for $\cE(\cdot,\cdot)$), see Chapter 1 \S5 of \cite{FukuOshiTake94}.

\medskip
We are now ready for the key identity of this section. The arguments we use are general and easily adapted to the context of continuous-time random interlacements on transient weighted graphs, as explained in Remark \ref{rem2.3} below.

\begin{theorem}\label{theo2.2}
\begin{equation}\label{2.5}
\Lambda(V) = \Gamma(V), \;\; \mbox{for all}\;\;  V \in L^\infty_c(\IR^d).
\end{equation}
\end{theorem}

\begin{proof} We first assume $\Lambda(V) < \infty$ and show the identity (\ref{2.5}). By (\ref{2.1}), we know that (\ref{1.10}), and hence (\ref{1.6}) hold. By (\ref{2.2}) we have
\begin{equation}\label{2.6}
\Lambda (V) = \langle V, \gamma_V\rangle \stackrel{(\ref{1.12})}{=} \langle V, 1\rangle + \langle V, G_V V\rangle.
\end{equation}
Similarly, we have
\begin{align}
\Gamma(V) & = \langle V,1\rangle + \Psi(V), \; \mbox{where} \label{2.7}
\\[1ex]
\Psi(V) & = \sup\limits_{\varphi \in L^2(\IR^d)} \{ 2 \langle V,\varphi \rangle + \langle V \varphi, \varphi\rangle - \cE(\varphi, \varphi)\} \nonumber
\\ 
& = \sup\limits_{\varphi \in L^2(\IR^d)} \sup\limits_{\ve > 0} \{2 \langle V, \varphi \rangle + \langle (V-\ve) \varphi, \varphi \rangle - \cE(\varphi, \varphi)\} \nonumber
\\[1ex]
& = \sup\limits_{\ve > 0} \; \sup\limits_{\varphi \in L^2(\IR^d)} \{ 2 \langle V, \varphi \rangle + \langle (V - \ve)\varphi, \varphi \rangle - \cE(\varphi,\varphi)\}.\nonumber
\end{align}

\n
We know by (\ref{1.8}) that the Schr\"odinger semi-group $(R^V_t)_{t \ge 0}$ is a semi-group of self-adjoint contractions on $L^2(\IR^d)$. Its quadratic form is $\cE(\varphi,\varphi) - \langle V \varphi,\varphi \rangle$, see for instance \cite{Chen02}, p.~4654. Then, by Lemma 4.4, p.~22 of \cite{Szni98a}, and also below (4.10), p.~23 of the same reference, we have
\begin{equation}\label{2.8}
\sup\limits_{\varphi \in L^2(\IR^d)} \{2 \langle V, \varphi \rangle + \langle (V-\ve)\varphi,\varphi \rangle - \cE(\varphi,\varphi)\} = \langle V, G_{V-\ve}V\rangle,
\end{equation}

\n
where $G_{V - \ve}$ is defined as in (\ref{1.4}) with $V$ replaced by $V-\ve$. Note that the left-hand side of (\ref{2.8}) is a decreasing function of $\ve$, so the same holds for the right-hand side. As a result, we find that
\begin{equation}\label{2.9}
\Psi(V) = \sup\limits_{\ve > 0} \langle V,G_{V-\ve}V\rangle = \lim\limits_{\ve \r 0} \langle V, G_{V-\ve}V\rangle = \langle V, G_V V\rangle,
\end{equation}

\n
using dominated convergence and (\ref{1.7}) in the last step. We have thus shown that
\begin{equation*}
\Gamma(V) \stackrel{(\ref{2.7})}{=} \langle V,1\rangle + \Psi (V)\stackrel{(\ref{2.9})}{=} \langle V,1\rangle + \langle V,G_V V\rangle \stackrel{(\ref{2.6})}{=} \Lambda(V),
\end{equation*}
that is, (\ref{2.5})  holds when $\Lambda(V) < \infty$.

\medskip
We now assume that $\Gamma(V) < \infty$, and show (\ref{2.5}). Since $\Gamma(V) < \infty$, we have 
\begin{equation}\label{2.10}
\sup\limits_{\varphi \in L^2(\IR^d)} \langle V \varphi, \varphi \rangle - \cE(\varphi,\varphi) \le 0
\end{equation}

\n
(otherwise the supremum in (\ref{2.4}) would be infinite). Since $\cE(\varphi,\varphi) - \langle V \varphi,\varphi\rangle$ is the quadratic form associated to the strongly continuous self-adjoint Schr\"odinger semi-group $(R^V_t)_{t \ge 0}$ on $L^2(\IR^d)$, it follows that $R^V_t,t \ge 0$, are contractions.

\medskip
We first discuss the case where $V_- = 0$ a.e.. If $V = 0$ a.e., then (\ref{2.5}) is immediate. Otherwise, $V_+$ is not a.e.\,equal to $0$. We then apply the same considerations as below (\ref{2.7}) to find that 
for $\ve>0$, $\sup_{\varphi \in L^2(\IR^d)} \{2 \langle V, \varphi \rangle + \langle (V-\ve)\varphi,\varphi \rangle - \cE(\varphi,\varphi)\} = \langle V, G_{V-\ve}V\rangle$, and that this quantity increases to $\sup_{\varphi \in L^2(\IR^d)} \{2 \langle V, \varphi \rangle + \langle V \varphi,\varphi \rangle - \cE(\varphi,\varphi)\}$, as $\ve$ tends to $0$. Coming back to the definition of $\Gamma(V)$ in (\ref{2.4}), we find that

\begin{equation}\label{2.11}
\infty > \Gamma(V) = \langle V,1\rangle + \lim\limits_{\ve \r 0} \langle V,G_{V-\ve}V \rangle = \langle V,1\rangle + \langle V,G_V V\rangle,
\end{equation}

\n
where we used monotone convergence in the last step. Hence, (\ref{1.6}) holds and $\Lambda(V) < \infty$, by (\ref{2.2}) and (\ref{1.11}). The identity (\ref{2.5}) follows from the first part of the proof.

\medskip
If instead $V_-$ is not a.e.\,equal to $0$, we define for $\lambda \in [0,1]$,
\begin{equation}\label{2.12}
V_\lambda = (1-\lambda) V-\lambda V_- = V - \lambda V_+ \;\big(\in L^\infty_c(\IR^d)\big),
\end{equation}

\n
so that $V_\lambda$ increases to $V$ as $\lambda$ decreases to $0$. Note that
\begin{align}
&2 \langle V, \varphi \rangle + \langle V \varphi, \varphi\rangle - \cE(\varphi,\varphi) \ge 2\langle V_\lambda, \varphi \rangle + \langle V_\lambda \varphi, \varphi \rangle - \cE(\varphi, \varphi) + A_\lambda, \; \mbox{where}\label{2.13}
\\[1ex]
&A_\lambda = \inf\limits_\varphi \lambda \{2 \langle V_+, \varphi\rangle + \langle V_+ \varphi,\varphi\rangle\} \stackrel{\rm Cauchy-Schwarz}{\ge} \lambda \inf\limits_{u \ge 0} \{ - 2\langle V_+,1\rangle^{\frac{1}{2}} u + u^2\} > - \infty \nonumber
\end{align}
(in the second line, we used the Cauchy-Schwarz Inequality in $L^2(V_{+} dy)$ to write that $\langle V_+, \varphi\rangle \ge -\langle V_+,1\rangle^{\frac{1}{2}} \langle V_+\varphi,\varphi\rangle^{\frac{1}{2}}$, and took a lower bound over $u=\langle V_+\varphi,\varphi\rangle^{\frac{1}{2}}$).

\medskip\n
Since $\Gamma(V) < \infty$ by assumption, we see that
\begin{equation}\label{2.14}
\Gamma(V_\lambda) < \infty, \; \mbox{for} \; \lambda \in [0,1].
\end{equation}
Moreover, for $0 < \lambda \le 1$, we have
\begin{align}
&2 \langle V, \varphi \rangle + \langle V \varphi, \varphi\rangle - \cE(\varphi,\varphi) \ge -2\langle V_-, \varphi \rangle + \langle V_\lambda \varphi, \varphi \rangle - \cE(\varphi, \varphi) + B_\lambda, \; \mbox{where} \label{2.16}
\\[1ex]
&B_\lambda = \inf\limits_\varphi 2 \langle V_+, \varphi \rangle + \lambda \langle V_+ \varphi,\varphi \rangle  > - \infty, \;\mbox{by a similar argument as below (\ref{2.13})}. \nonumber
\end{align}
This shows that
\begin{equation}\label{2.17}
\infty > \sup\limits_{\varphi \in L^2(\IR^d)} \{ - 2 \langle V_-, \varphi \rangle + \langle V_\lambda \varphi, \varphi\rangle - \cE(\varphi,\varphi)\}.
\end{equation}

\n
In addition, by (\ref{2.14}) and the argument below (\ref{2.10}), the Schr\"odinger semi-group $(R^{V_\lambda}_t)_{t \ge 0}$ is a strongly continuous semi-group of self-adjoint contractions on $L^2(\IR^d)$. From the argument below (\ref{2.7}), we see that the above supremum in (\ref{2.17}) equals $\lim_{\ve \r 0} \langle V_-, G_{V_\lambda - \ve} V_-\rangle = \langle V_-, G_{V_\lambda} V_-\rangle$ (using monotone convergence for the last equality), and this quantity is finite. Hence, $V_\lambda$ satisfies (\ref{1.6}) and (see below (\ref{2.11})) $\Lambda (V_\lambda) < \infty$. So, by the first part of the proof, 
\begin{equation}\label{2.18}
\Lambda(V_\lambda) = \Gamma(V_\lambda) ( < \infty), \; \mbox{for} \; 0 < \lambda \le 1.
\end{equation}

\medskip\n
By monotone convergence in (\ref{2.3}), we see that
\begin{equation}\label{2.19}
\lim\limits_{\lambda \downarrow 0} \Lambda(V_\lambda) = \Lambda (V).
\end{equation}

\n
On the other hand, by (\ref{2.4}) and (\ref{2.12}), $\Gamma(V_\lambda)$ is a supremum of affine functions of $\lambda \in [0,1]$, which is finite when $\lambda = 1$ (since $V_{\lambda = 1} \le 0$) and when $\lambda = 0$ (by assumption). Hence, it is a convex, lower semi-continuous, finite function on $[0,1]$, which is therefore continuous, so that
\begin{equation}\label{2.20}
\lim\limits_{\lambda \r 0} \Gamma(V_\lambda) = \Gamma (V).
\end{equation}

\medskip\n
This implies that (\ref{2.5}) holds and completes the proof of Theorem \ref{theo2.2}.
\end{proof}

\begin{remark}\label{rem2.3} \rm The proof of Theorem \ref{theo2.2} is easily adapted to the case of continuous time random interlacements on a transient weighted graph $E$. In this set-up, see for instance Section 1 of \cite{Szni12b}, one has a countable, locally finite, connected graph, with vertex set $E$, endowed with non-negative symmetric weights $c_{x,y} = c_{y,x}$, for $x,y \in E$, which are positive exactly when $x,y$ are distinct and $\{x,y\}$ is an edge of the graph. The induced continuous time random walk is assumed to be transient. It has exponential holding times of parameter $1$, and its discrete skeleton has transition probability
\begin{equation}\label{2.21}
p_{x,y} = \dis\frac{c_{x,y}}{\lambda_x}, \;\; \mbox{where}\;\;\lambda_x = \dsl_{z \in E} c_{x,z}, \;\; \mbox{for}\;\; x,y \in E.
\end{equation}

\n
The continuous time random interlacements on the weighted graph can now be defined as a Poisson point process on a space of doubly-infinite $E$-valued trajectories, tending to infinity at plus and minus infinite times, marked by their duration at each step, modulo time-shift, see Section 1 of \cite{Szni12b}. The field of occupation-times at level $u \ge 0$ in $x \in E$ corresponds to 
\begin{equation}\label{2.22}
\begin{split}
L_{x,u}(\o) = \mbox{\f $\dis\frac{1}{\lambda_x}$} \times &\; \mbox{the total duration spent at $x$ by trajectories modulo}
\\[-1ex]
&\; \mbox{time-shift, with label at most $u$ in the cloud $\omega$}
\end{split}
\end{equation}

\n
(durations of successive steps of a trajectory are described by independent exponential variables of parameter $1$, but occupation times at $x$ get rescaled by $\lambda_x^{-1}$).

\medskip
In this set-up, one introduces for $V$: $E \r \IR$, finitely supported
\begin{align}
\Lambda(V)  = &\; \mbox{\f $\dis\frac{1}{u}$} \;\ov{\IE} \Big[\exp\Big\{\dsl_{x \in E} L_{x,u} V(x)\Big\}\Big] \; \mbox{(this does not depend on $u > 0$ by the}  \label{2.23}
\\
&\;\mbox{corresponding calculation to (\ref{2.1})), and} \nonumber
\\[2ex]
\Gamma(V)  =&\; \langle V,1\rangle + \sup\limits_{\varphi \in C_c(E)} \big\{2\langle V,\varphi \rangle + \langle V \varphi,\varphi\rangle - \cE(\varphi,\varphi)\big\}, \label{2.24}
\end{align}

\n
where $\langle f,g\rangle$ stands for $\sum_{x \in E} f(x)\,g(x)$ whenever the sum converges absolutely, $1$ denotes the constant function equal to $1$ on $E$, $C_c(E)$ stands for the set of finitely supported functions on $E$, and $\cE$ for the Dirichlet form
\begin{equation}\label{2.25}
\cE(f,f) = \fr \;\dsl_{x,y \in E} c_{x,y} \big(f(y) - f(x)\big)^2 \,(\in [0,\infty]) \; \mbox{for $f$: $E \r \IR$}.
\end{equation}

\n
Let us mention that one can also replace $C_c(E)$ in (\ref{2.24}) by the extended space of the Dirichlet form $\cE$ (see Chapter 1 \S 5 of \cite{FukuOshiTake94}). The proof of Theorem \ref{theo2.2} adapted to the present context yields

\bigskip\n
{\bf Theorem 2.2'.} {\it
\begin{equation}\label{2.26}
\mbox{$\Lambda(V) = \Gamma(V)$, for all finitely supported $V$: $E \r \IR$}.
\end{equation}
}
\hfill $\square$
\end{remark}

We now derive some corollaries of Theorem \ref{theo2.2}, which will play an important role in the next section. As above (\ref{0.1}), $B$ is a closed box (i.e. a compact subset of $\IR^d$ that is the product of $d$ possibly different non-degenerate compact intervals in $\IR$), and $m_B$, as above (\ref{0.2}), the restriction of Lebesgue measure to $B$. We write $L^p(B)$ as a shorthand for $L^p(B,dm_B)$, when $1 \le p < \infty$. Given $\wt{\varphi} \in L^2(B)$, we introduce 
\begin{equation}\label{2.27}
\wt{\cE}_B(\wt{\varphi},\wt{\varphi}) = \inf\limits_{\varphi \in L^2(\IR^d)} \,\{\cE(\varphi,\varphi); \;\varphi = \wt{\varphi}\;  \mbox{~a.e.\,on} \;B\},
\end{equation}

\n
the so-called {\it trace Dirichlet form} on $B$, see \cite{FukuOshiTake94}, pp.~265, 266. We will often drop the subscript $B$, when this causes no ambiguity. One knows that $\wt{\cE}_B(\wt{\varphi}, \wt{\varphi})$ is finite precisely when $\wt{\varphi}$ is a.e.\,equal to the restriction to $B$ of a quasi-continuous function $\ov{\varphi}$ in the extended Dirichlet space $\cF_e$, and in this case $\varphi(y) = E_y[\ov{\varphi}(X_{H_B})$, $H_B < \infty]$ belongs to $\cF_e$ and $\widetilde{\cE}_B(\wt{\varphi},\wt{\varphi}) = \cE(\varphi,\varphi)$.

\medskip
Let us also note that $L^1_+(B) \stackrel{\rm def}{=} \{f \in L^1(B)$; $f \ge 0$, $m_B$-a.e.$\}$ is a closed convex subset of $L^1(B)$ (endowed with the norm topology).

\begin{corollary}\label{cor2.4}
For $V \in L^\infty_c (\IR^d)$ vanishing outside $B$, one has
\begin{align}
&\Lambda (V) = \sup\limits_{h \in L^1_+(B)} \Big\{ \dis\int_B Vh\, dm_B - \wt{\cE} (\sqrt{h} - 1, \sqrt{h} -1)\Big\}, \;\mbox{and} \label{2.28}
\\[1ex]
&h \in L^1_+(B) \longrightarrow \wt{\cE} (\sqrt{h} - 1, \sqrt{h} - 1) \in [0, + \infty] \;\mbox{is a convex lower} \label{2.29}
\\
&\mbox{semi-continuous function}. \nonumber
\end{align}
\end{corollary}

\begin{proof}
We consider $V \in L^\infty_c(\IR^d)$ vanishing outside $B$. Then, by (\ref{2.5}) we have
\begin{equation}\label{2.30}
\begin{split}
\Lambda(V) &=  \sup\limits_{\varphi \in L^2(\IR^d)} \Big\{\dis\int V \,dy + 2 \dis\int V\varphi\,dy + \dis\int V \varphi^2 dy - \cE(\varphi,\varphi)\Big\}
\\[1ex]
& =  \sup\limits_{\varphi \in L^2(\IR^d)} \Big\{\dis\int V (1 + \varphi)^2 dy - \cE(\varphi,
\varphi)\Big\}
\\[1ex]
&\!\!\! \stackrel{(\ref{2.27})}{=}   \sup\limits_{\wt{\varphi} \in L^2(B)} \Big\{\dis\int_B V(1 + \wt{\varphi})^2 dm_B - \wt{\cE}(\wt{\varphi},\wt{\varphi})\Big\} \; \mbox{(since $V = 0$ outside $B$).}
\end{split}
\end{equation}

\n
Now, the function $\rho(u) = |1 + u| -1$ satisfies $\rho(0) = 0$ and $|\rho(u) - \rho(v)| \le |u - v|$, for $u,v \in \IR$, and by \cite{FukuOshiTake94}, pp.~4, 5, we have for $\wt{\varphi} \in L^2(B)$, $\wt{\psi} = \rho(\wt{\varphi})$, $\wt{\cE}\,(\wt{\psi},\wt{\psi}) \le \wt{\cE} (\wt{\varphi},\wt{\varphi})$, and $(1 + \wt{\psi})^2 = (1 + \wt{\varphi})^2$. As a result, we find that
\begin{equation}\label{2.31}
\Lambda(V) = \sup\Big\{\dis\int V(1 + \wt{\psi})^2 dm_B - \wt{\cE}(\wt{\psi},\wt{\psi}); \;\wt{\psi} \in L^2(B), \; 1 + \wt{\psi} \ge 0 \; \mbox{~a.e.}\Big\}
\end{equation}

\n
(the above argument shows that $\Lambda(V)$ is smaller or equal to the right-hand side of (\ref{2.31}), but by the last line of (\ref{2.30}), $\Lambda(V)$ is also bigger or equal to the right-hand side of (\ref{2.31})).

\medskip
Setting $h = (1 + \wt{\psi})^2$, we obtain a bijection between $\{\wt{\psi} \in L^2(B)$; $1 + \wt{\psi} \ge 0$, a.e.$\}$, and $L^1_+(B)$, and the claim (\ref{2.28}) follows. As for (\ref{2.29}), it is proved by a similar argument as in ii), below (4.2.64), p.~135 of \cite{DeusStro89} (see also Theorem 6.2.1 and (1.3.18), (1.4.8) in \cite{FukuOshiTake94}).
\end{proof}

The next corollary brings us a step closer to the identification of what will be the rate function of the large deviation principle, which we derive in the next section. We tacitly identify the set $L^\infty(B)$ of bounded measurable function on $B$ with the set $\{V \in L^\infty_c (\IR^d)$; $V = 0$ on $B^c\}$, and recall that (see below (\ref{0.1})) $C(B)$ stands for the space of continuous functions on $B$ (identified with the set of functions vanishing outside $B$, with continuous restriction to $B$).

\begin{corollary}\label{cor2.5}
For $h \in L^1_+(B)$, one has
\begin{equation}\label{2.32}
\begin{split}
\wt{\cE}_B(\sqrt{h} - 1, \sqrt{h} - 1) & = \sup\limits_{V \in L^\infty(B)} \Big\{\dis\int V h\, dm_B - \Lambda(V)\Big\}
\\[1ex]
& = \sup\limits_{V \in C(B)} \Big\{\dis\int V h\, dm_B - \Lambda (V)\Big\}.
\end{split}
\end{equation}
\end{corollary}

\medskip
\begin{proof}
Since $L^\infty(B,dm_B)$ is the dual of $L^1(B,dm_B)$, the first equality follows from (\ref{2.28}), (\ref{2.29}) and the duality formula in Theorem 2.2.15, p.~55 of \cite{DeusStro89}, or Lemma 4.5.8, p.~152 of \cite{DembZeit98}, together with the fact we can replace $L^\infty(B,dm_B)$ by $L^\infty(B)$, since the quantity under the supremum in the right-hand side of the first equality coincides for $V$ and $V'$, if $V = V'$, $m_B$-a.e.\,(recall (\ref{1.24})).

\medskip
The quantity on the last line of (\ref{2.32}) is obviously smaller or equal to the right-hand side of the first equality. Fix $h$ in $L^1_+(B)$. Our claim will thus follow, once we show that for any $V \in L^\infty(B)$ with $\Lambda(V) < \infty$ and $\ve > 0$, one can find $W \in C(B)$ such that
\begin{equation}\label{2.33}
\dis\int W h \,dm_B - \Lambda(W) \ge \dis\int V h\, dm_B - \Lambda(V) - \ve.
\end{equation}

\n
We thus pick $r > \frac{d}{2}$, as in Lemma \ref{lem1.3}, and choose $V_n \in C(B)$ (extended to be equal to $0$ on $B^c$) such that
\begin{equation}\label{2.34}
\left\{ \begin{array}{rl}
{\rm i)} &  \|V_n - V\|_{L^r(\IR^d)} \r 0, \hspace{8cm}
\\[2ex]
{\rm ii)} & V_n \r V\; \mbox{a.e.},
\\[2ex]
{\rm iii)}&  \|V_n\|_{L^\infty(\IR^d)} \le \|V\|_{L^\infty(\IR^d)} .
\end{array}\right.
\end{equation}

\medskip\n
The construction of such a sequence can be performed as follows: one first multiplies $V$ by the indicator function of a slightly smaller concentric box in $B$, and then uses convolution by a smoothing kernel to construct a sequence $V_n$ of continuous functions vanishing on $B^c$, for which i) and iii) hold. One obtains ii) by extracting a suitable subsequence. Then Lemma \ref{lem1.3} and (\ref{2.2}) ensure that $\Lambda(V_n) \underset{n}{\longrightarrow} \Lambda(V)$, and moreover $\int V_n h\, dm_B \underset{n}{\longrightarrow} \int Vh\, dm_B$, by dominated convergence. The claim (\ref{2.33}) follows, and Corollary \ref{cor2.5} is proved.
\end{proof}

\section{Large deviations for occupation-time profiles of  \newline Brownian interlacements}
\setcounter{equation}{0}

In this section, we derive a large deviation principle for the occupation-time profile of Brownian interlacements at level $\alpha > 0$ in a box $LB$, as $L \r \infty$, cf.~Corollary \ref{cor3.3}. Due to the scaling property (\ref{1.25}) of $\cL_\alpha$, this fixed level, large space problem (i.e.~$\alpha > 0$ fixed, and $L \r \infty$) is converted into a fixed space, large level problem (i.e.~$L$ fixed, and $\alpha \r \infty$), which can be handled via general Cram\'er theory, see Chapter 6 \S 1 of \cite{DembZeit98}, or Chapter 3 of \cite{DeusStro89}, making use of subadditivity, see Theorem \ref{theo3.2}. The identification of the rate function relies heavily on the results of the previous section, in particular on Corollary \ref{cor2.5}.

\medskip
Given a closed box $B$ in $\IR^d$ (see above (\ref{0.1})), we endow the space $M(B)$ of finite signed measures on $B$ with the weak topology generated by $C(B)$ (the space of continuous functions on $B$). The set $M_+(B)$ (of positive measures on $B$) is a closed convex subset of $M(B)$. We introduce the function on $M_+(B)$ (see (\ref{2.3}) for notation):
\begin{equation}\label{3.1}
I_B(\mu) = \sup\limits_{V \in C(B)} \Big\{\dis\int V \,d\mu - \Lambda(V)\Big\} \in [0,\infty], \; \mbox{for} \; \mu \in M_+(B),
\end{equation}

\n
and when there is no ambiguity, we simply write $I(\cdot)$ in place of $I_B(\cdot)$.

\medskip
As we  now see, $I_B(\cdot)$ is closely related to $I_v(\cdot)$, $v > 0$, in (\ref{0.3}).

\begin{lemma}\label{lem3.1}  ($B \subseteq B'$ closed boxes in $\IR^d$)
\begin{align}
&~\mbox{$I_B(\cdot)$ is a convex, good rate function (i.e.,~it is convex, lower semi-continuous,}  \label{3.2}
\\
&~\mbox{and has compact level sets).} \nonumber
\end{align}

\vspace{-5ex}
\begin{align}
I_B(\mu) \le &\; I_{B'}(\mu') \; \mbox{if $\mu$ is the restriction to $B$ of $\mu' \in M_+(B')$}.\label{3.3}
\\[2ex]
I_B(\mu) = & \; + \infty, \;\mbox{if $\mu \in M_+(B)$ is not absolutely continuous with respect to $m_B$} \label{3.4}
\\
 =& \; \inf\Big\{\fr \;\dis\int_{\IR^d} | \nabla \varphi|^2 dy; \;\varphi \in H^1(\IR^d), \; \varphi = \Big(\mbox{\f $\dis\frac{d \mu}{dm_B}$}\Big)^{\frac{1}{2}} - 1 \; \mbox{a.e.\,on} \;B\Big\}, \nonumber
\\
& \;\mbox{if $\mu$ is absolutely continuous with respect to $m_B$}.\nonumber
\\[2ex]
I_v(\cdot)  = & \;\mbox{$v\,I_B\Big(\dis\frac{\cdot}{v}\Big)$, for $v > 0$ (see (\ref{0.3}) for notation)}. \label{3.5}
\end{align}
\end{lemma}

\begin{proof}
We begin with (\ref{3.2}). The convexity and lower semi-continuity are direct consequences of (\ref{3.1}). Moreover, when $\lambda > 0$ is small enough so that $\lambda \| G1_B\|_{L^\infty(\IR^d)} < 1$, we know that $\Lambda (\lambda 1_B) < \infty$, see (\ref{1.26}), and hence,
\begin{equation}\label{3.6}
I(\mu) \ge \lambda \langle \mu,1_B\rangle - \Lambda (\lambda 1_B), \;\mbox{for $\mu \in M_+ (B)$}.
\end{equation}

\n
Therefore, for any $M \ge 0$, $\langle \mu, 1_B\rangle$ remains bounded on $\{ \mu \in M_+(B)$; $I(\mu) \le M\}$, which is therefore a compact subset of $M_+(B)$ (we already know it is closed by lower semi-continuity of $I$). This proves (\ref{3.2}).

\medskip
We now turn to (\ref{3.4}). The second line of (\ref{3.4}) is an immediate consequence of Corollary \ref{cor2.5} and (\ref{2.27}). To complete the proof of (\ref{3.4}), it thus suffices to show that $I(\mu) = \infty$ when $\mu$ is not absolutely continuous with respect to $m_B$.

\medskip
Indeed, for such a measure $\mu$, we can find a compact subset $K$ of $B$ such that $\mu(K) > 0$, but $m(K) = 0$. Then, consider for $n \ge 1$, $W_n(\cdot) = (1 - nd(\cdot,K))_+ \in C(B)$ (we set $W_n$ equal to zero on $B^c$), so that $W_n$ is $[0,1]$-valued, takes the values $1$ on $K$ and vanishes outside the open $\frac{1}{n}$-neighborhood of $K$ in $B$. Thus, choosing $r > \frac{d}{2}$, we see that $\delta_n = \|W_n\|_{L^r(\IR^d)} \underset{n}{\longrightarrow} 0$, and setting $V_n = \delta_n^{-\frac{1}{2}} W_n$, it follows by Lemma \ref{lem1.3} and (\ref{2.2}) that $\Lambda (V_n) \underset{n}{\longrightarrow} 0$. As a consequence, we see that  
\begin{equation}\label{3.7}
I(\mu) \stackrel{(\ref{3.1})}{\ge} \dis\int V_n d \mu - \Lambda(V_n) \ge \delta_n^{-\frac{1}{2}} \mu(K) - \Lambda (V_n) \underset{n}{\longrightarrow} + \infty.
\end{equation}

\n
This shows that $I(\mu) = \infty$ and completes the proof of (\ref{3.4}).

\medskip
We now prove (\ref{3.3}). Without loss of generality, we can assume that $I_{B'}(\mu') < \infty$. By (\ref{3.4}) this implies that $\mu'$ is absolutely continuous with respect to $m_{B'}$, and hence $\mu$ is absolutely continuous with respect to $m_B$. The inequality $I_B(\mu) \le I_B(\mu')$ is now a direct consequence of (\ref{3.4}). This proves (\ref{3.3}). Finally, (\ref{3.5}) is immediate by comparison of (\ref{3.4}) and (\ref{0.3}).
\end{proof}

We denote by $\cL_{\alpha,B}$ the restriction of the (random) Radon measure $\cL_\alpha$ on $\IR^d$ to the Borel subsets of $B$. The main step towards the large deviation principle for $\nu_{L,\alpha}$, with $\alpha > 0$ fixed and $L \r \infty$, see (\ref{0.5}), is a large deviation principle for $\frac{1}{\alpha} \,\cL_{\alpha,B}$, as $\alpha \r \infty$, which we derive in Theorem \ref{theo3.2} below. We use general Cram\'er theory, see Theorem 6.1.3, p.~252 of \cite{DembZeit98} (with $\chi = M(B)$ and $\cE = M_+ (B)$, in the notation of \cite{DembZeit98}).

\begin{theorem}\label{theo3.2} ($B$ a closed box)
\begin{equation}\label{3.8}
\begin{array}{l}
\mbox{As $\alpha \r \infty$, the laws under $\IP$ of $\frac{1}{\alpha} \;\cL_{\alpha,B}$ on $M_+(B)$ satisfy a large deviation}
\\
\mbox{principle at speed $\alpha$, with the convex good rate function $I_B$ from (\ref{3.1})}.
\end{array}
\end{equation}
Moreover, for any open convex subset $O$ of $M(B)$,
\begin{equation}\label{3.9}
\lim\limits_{\alpha \r \infty} \; \mbox{\f $\dis\frac{1}{\alpha}$} \; \log \IP\Big[\mbox{\f $\dis\frac{1}{\alpha}$} \;\cL_{\alpha, B} \in O\Big] = - \inf\{I_B(\mu); \; \mu \in O \cap M_+(B)\}
\end{equation}
(with $\inf \phi = \infty$, by convention).
\end{theorem}

\begin{proof}
For any open convex subset $O$ of $M(B)$, the function
\begin{equation}\label{3.10}
f_O(\alpha) = - \log \IP \Big[  \mbox{\f $\dis\frac{1}{\alpha}$} \; \cL_{\alpha,B} \in O\Big] \in [0,+ \infty], \; \mbox{for $\alpha > 0$},
\end{equation}

\medskip\n
is subadditive (since $\cL_{\alpha, B}$ is an $M_+(B)$-valued L\'evy-process). When $O \cap M_+(B) = \phi$, then $f_O$ is identically infinite. Otherwise, when $O \cap M_+(B) \not= \phi$, then, as we now explain, for $\ve \in (0,1]$ small,
\begin{equation}\label{3.11}
\sup\limits_{1 \le \alpha \le 1 + \ve} f_O(\alpha) < \infty.
\end{equation}

\n
Indeed, for small $\ve > 0$, there is a positive $\IP$-probability that exactly one trajectory of the interlacement at level $1$ enters $B$ (i.e. $\mu_{B,\alpha = 1} = 1$, in the notation of (\ref{1.20})) and that $\cL_{1,B} \in \alpha O$, for all $1 \le \alpha \le 1 + \ve$ (this can be arranged, using the fact that $O$ is open, $O \cap M_+(B) \not= \phi$, with the help of the support theorem for the Wiener measure, and the observation that once reaching distance $1$ from $B$, a Brownian trajectory has a non-degenerate probability of never returning to $B$). Since $\cL_{1 + \ve,B} - \cL_{1,B}$ is independent from $\cL_{1,B}$, and vanishes with positive probability, we see that on an event of positive $\IP$-probability $\frac{1}{\alpha} \,\cL_{\alpha,B} \in O$, for $1 \le \alpha \le 1 + \ve$. This is more than enough to prove (\ref{3.11}).

\medskip
Hence, by Lemma 4.2.5, p.~112 of \cite{DeusStro89}, we see that when $O \cap M_+(B) \not= \phi$,
\begin{equation}\label{3.12}
\mbox{$\lim\limits_{\alpha \r \infty} \; \mbox{\f $\dis\frac{1}{\alpha}$} \;f_O(\alpha)$ exists and is finite (and equals $\inf\limits_{\alpha \ge \alpha_0} \;\mbox{\f $\dis\frac{f_O(\alpha)}{\alpha}$}$ for large $\alpha_0$).}
\end{equation}

\n
Moreover, choosing $\lambda$ as above (\ref{3.6}), the Chebyshev Inequality and (\ref{1.26}) imply that 
\begin{equation*}
\IP[\cL_\alpha(B) \ge \alpha M] \le \exp\big\{ - \alpha \big(\lambda M - \Lambda(\lambda 1_B)\big)\big\}, \; \mbox{for $M > 0$, $\alpha > 0$}.
\end{equation*}

\medskip\n
Since $\{\mu \in M_+(B)$; $\mu(B) \le M\}$ is compact for each $M > 0$, and $\lambda M - \Lambda  (\lambda 1_B) \underset{M}{\longrightarrow} \infty$, the exponential tightness of the laws of $\frac{1}{\alpha} \;\cL_{\alpha,B}$ follows. Thus, by Theorem 4.1.11, p.~120 of \cite{DembZeit98}, and p.~8 of the same reference, we see that, as $\alpha \r \infty$, the laws of $\frac{1}{\alpha} \,\cL_{\alpha,B}$ satisfy a large deviation principle at speed $\alpha$. Restricting $\alpha$ to integer values, it follows from Theorem 6.1.3, p.~252 of \cite{DembZeit98} (Assumption 6.1.2 is straightforward to check in our set-up) that the rate function of the above large deviation principle coincides with $I_B$ in (\ref{3.1}) and that the limit in (\ref{3.12}) coincides with the right-hand side of (\ref{3.9}) (when $O \cap M_+(B) = \phi$, both members of (\ref{3.12}) equal $+ \infty$). This concludes the proof of Theorem \ref{theo3.2}.
\end{proof}

The large deviation principle for the profile $\nu_{L,\alpha}$ of the occupation-time measure of Brownian interlacements at level $\alpha$ in a large box $LB$ is now a direct consequence of Theorem \ref{theo3.2} and the scaling property (\ref{1.25}). We refer to (\ref{0.3}), (\ref{0.5}), and above (\ref{0.1}) for notation.

\begin{corollary}\label{cor3.3} (Large deviation principle for the profile, $B$ a closed box, $\alpha > 0$)
\begin{equation}\label{3.13}
\begin{array}{l}
\mbox{As $L \r \infty$, the laws under $\IP$ of $\nu_{L,\alpha}$ on $M_+(B)$ satisfy a large deviation}
\\
\mbox{principle with speed $L^{d-2}$ and rate function $I_\alpha$}.
\end{array}
\end{equation}
Moreover, for any open convex subset $O$ of $M(B)$,
\begin{equation}\label{3.14}
\lim\limits_{L \r \infty} \;\mbox{\f $\dis\frac{1}{L^{d-2}}$} \;\log \IP [\nu_{L,\alpha} \in O] = -\inf\{I_\alpha(\mu); \;\mu \in O \cap M_+(B)\}
\end{equation}
(with $\inf \phi = + \infty$ by convention).
\end{corollary}

\begin{proof}
By the scaling property (\ref{1.25}), we see that $\nu_{L,\alpha}$ has the same law as $\wt{\nu}_L =$ \linebreak $\frac{1}{L^{d-2}}$ $\cL_{\alpha\, L^{d-2},B}$, for any $L \ge 1$. Our claims are now direct consequences of Theorem~\ref{theo3.2}.
\end{proof}

\section{The discrete space set-up}
\setcounter{equation}{0}

In this section, we introduce some additional notation concerning continuous time random interlacements on $\IZ^d$, $d \ge 3$. We recall the scaling limit relating the discrete space occupation-times to the occupation-time measure of Brownian interlacements established in Theorem 3.2 of \cite{Szni13a}, see (\ref{4.4}) below. Further, as a preparation to the large deviation principle for $\rho_{N,u}$ (see (\ref{0.1})), which we derive in the next section, we collect here the statements corresponding to Theorem 2.2' and Corollaries \ref{cor2.4}, \ref{cor2.5} in the present set-up.

\medskip
Given $N \ge 1$, we introduce the scaled lattice 
\begin{equation}\label{4.1}
\IL_N = \mbox{\f $\dis\frac{1}{N}$} \;\IZ^d (\subseteq \IR^d),
\end{equation}

\n
and $B$ being a closed box (see above (\ref{0.1})), we set
\begin{equation}\label{4.2}
B_N = B \cap \IL_N .
\end{equation}

\n
For functions $f,h$ on $\IL_N$ such that $\sum\limits_{y\in \IL_N} |f(y)\,h(y)| < \infty$, we write
\begin{equation}\label{4.3}
\langle f,h\rangle_{\IL_N} = \mbox{\f $\dis\frac{1}{N^d}$} \;\dsl_{y \in \IL_N} f(y)\,h(y),
\end{equation}

\medskip\n
and we introduce the spaces $L^p(\IL_N)$, $1 \le p < \infty$, and their corresponding norms in a similar manner.

\medskip
We refer to \cite{Szni12b} for the precise construction of continuous time random interlacements on $\IZ^d$ (with $d \ge 3$). As mentioned in the introduction, we denote by $(\ov{\Omega}, \ov{\cA}, \ov{\IP})$ the canonical space on which they are constructed, and by $(L_{x,u})_{x \in \IZ^d}$ the field of occupation-times of random interlacements at level $u \ge 0$. Fixing $u \ge 0$, this is a stationary field on $\IZ^d$, and by the same proof as for (2.7) of \cite{Szni10a}, it is ergodic. One knows that $\ov{\IE}[L_{x,u}] = u$, for $x \in \IZ^d$, $u \ge 0$, so by the ergodic theorem, see Theorem 2.8, p.~205 of~\cite{Kren85}, one sees that
\begin{equation}\label{4.4}
\mbox{for $u \ge 0$, $\ov{\IP}$-a.s., $\mbox{\f $\dis\frac{1}{N^d}$} \dsl_{y \in \IL_N} L_{Ny,u} \,\delta_y$ converges vaguely to $u\,dy$, as $N \r \infty$}
\end{equation}

\n
(and $\rho_{N,u}$ converges weakly to $u\,m_B$, in the notation of (\ref{0.1})).

\medskip
The occupation-time measure of Brownian interlacements can be expressed as a scaling limit of the occupation times of random interlacements on $\IZ^d$. Namely, one know by Theorem 3.2 of \cite{Szni13a} that for $\alpha \ge 0$, as $N \r \infty$,
\begin{equation}\label{4.5}
\mbox{$\cL_{N,\alpha} \stackrel{\rm def}{=} \mbox{\f $\dis\frac{1}{dN^2}$} \; \dsl_{y\in \IL_N} L_{Ny,d\alpha N^{2-d}}\, \delta_y$ converges in distribution to $\cL_\alpha$}
\end{equation}

\n
(we endow the set of Radon measures on $\IR^d$ with the topology of vague convergence).

\medskip
For $N \ge 1$ and $V$: $\IL_N \r \IR$ with finite support, we introduce 
\begin{equation}\label{4.6}
\Lambda_N(V) = \log \ov{\IE} [e^{\langle \cL_{N,1},V\rangle}] \in (-\infty, + \infty],
\end{equation}
and note that
\begin{equation}\label{4.7}
\Lambda_N(V) = \mbox{\f $\dis\frac{d}{N^{d-2}}$}\,\log\ov{\IE}[e^{\frac{N^{d-2}}{d} \langle \rho_{N,u=1},V\rangle}], \;\mbox{when $V$ vanishes outside $B_N$}
\end{equation}

\n
(where we used (\ref{0.1}) and the fact that the right-hand side of (\ref{2.23}) does not depend on $u$).

\medskip
To state the identity corresponding to Theorem 2.2' in the present set-up, we define, for $N \ge 1$ and $V$: $\IL_N \r \IR$ finitely supported,
\begin{equation}\label{4.8}
\Gamma_N(V) = \langle V,1\rangle_{\IL_N} + \sup\limits_{\varphi \in L^2(\IL_N)} \{2 \langle V,\varphi \rangle_{\IL_N} + \langle V \varphi,  \varphi\rangle_{\IL_N} - \cE_N(\varphi,\varphi)\},
\end{equation}
where for $\varphi \in L^2(\IL_N)$
\begin{equation}\label{4.9}
\cE_N(\varphi,\varphi) = \mbox{\f $\dis\frac{1}{2N^{d-2}}$} \;\dsl_{y \sim y' \,{\rm in} \,\IL_N}\;\fr \;\big(\varphi(y') - \varphi(y)\big)^2 \in [0,\infty]
\end{equation}

\medskip\n
($y \sim y'$ means that $y$ and $y'$ are neighbors in $\IL_N$, i.e. $|y-y'| = \frac{1}{N}$).

\medskip
We also keep the notation $\cE_N(\varphi,\varphi)$ when $\varphi$ belongs to the extended space of the Dirichlet form $\cE_N$ (corresponding to functions on $\IL_N$ that are pointwise limits of an $\cE_N$-Cauchy sequence of finitely supported functions on $\IL_N$). After proper normalization (the $V$ in Theorem 2.2' corresponds to $\frac{1}{dN^2} \,V(\frac{\cdot}{N})$), Theorem 2.2' now yields

\begin{theorem}\label{theo4.1} $(N \ge 1)$
\begin{equation}\label{4.10}
\mbox{$\Lambda_N(V) = \Gamma_N(V)$, for all $V$: $\IL_N \r \IR$ with finite support}.
\end{equation}
\end{theorem}

Before stating the corollary corresponding to Corollaries \ref{cor2.4}, \ref{cor2.5} in the present set-up, we define for $\psi$: $B_N \r \IR$,
\begin{equation}\label{4.11}
\wt{\cE}_N(\psi,\psi) = \inf\limits\{\cE_N (\varphi,\varphi); \; \varphi \in L^2(L_N), \;\varphi = \psi \; \mbox{on} \;B_N\}.
\end{equation}

\n
Denoting by $P^N_y$, for $N \ge 1$, $y \in \IL_N$, the canonical law of the simple random walk on $\IL_N$ with exponential holding times of parameter $N^2$, starting at $y \in \IL_N$, and using similar notation for the canonical process, the entrance times, and the exit times, as described at the beginning of Section 1, one knows, as below (\ref{2.27}), that
\begin{equation}\label{4.12}
\wt{\cE}_N (\psi,\psi) = \cE_N(\varphi,\varphi) \in [0,\infty),
\end{equation}

\n
where $\varphi(y) = E^N_y [\psi(X_{H_{B_N}}),H_{B_N} < \infty]$, for $y \in \IL_N$, is harmonic outside $B_N$, tends to zero at infinity, and belongs to the extended Dirichlet space of $\cE_N$ (with $H_{B_N}$ the entrance time in $B_N$).

\medskip
Similarly to Corollaries \ref{cor2.4} and \ref{cor2.5} (but in a much simpler fashion), we have

\begin{corollary}\label{cor4.2} $(N > 1)$

\medskip
When $V$: $\IL_N \r \IR$ vanishes outside $B_N$
\begin{equation}\label{4.13}
\Lambda_N(V) = \sup\limits_{h \ge 0 \,{\rm on}\, B_N} \{\langle V,h\rangle_{\IL_N} - \wt{\cE}_N (\sqrt{h} - 1, \sqrt{h} - 1)\} \in [0,\infty],
\end{equation}

\medskip\n
and for $h$: $B_N \r \IR_+$ (extended to be equal to $0$ outside $B_N)$
\begin{equation}\label{4.14}
\wt{\cE}_N (\sqrt{h} - 1, \sqrt{h}-1) = \sup\limits_{V = 0 \,{\rm on}\,\IL_N \backslash B_N} 
\{ \langle V, h\rangle_{\IL_N} - \Lambda_N(V)\}.
\end{equation}
\end{corollary}

\section{Large deviations for occupation-time profiles of random interlacements}
\setcounter{equation}{0}

The main object of this section is to prove a large deviation principle for the occupation-time profile $\rho_{N,u}$, see (\ref{0.1}), of continuous time random interlacements on $\IZ^d$ at level $u$, when $N \r \infty$, cf.~Theorem \ref{5.8}. Subadditivity is an important ingredient in our proof, see Proposition \ref{prop5.2} and Corollary \ref{cor5.3}. The lower bound in the large deviation principle appears in Theorem \ref{theo5.4}. It relies on Corollary \ref{cor5.3}, and on the combination of the large deviation principle for the occupation-time profile of Brownian interlacements proved in Section 3, and the scaling limit result (\ref{4.5}) proved in \cite{Szni13a}, see also Lemma \ref{lem5.1} below. The upper bound appears in Theorem \ref{5.5}. It relies on Corollary \ref{cor5.3} and on Proposition \ref{prop5.6}, which provides an asymptotic lower bound for a sequence of discrete variational problems,  in the spirit of $\Gamma$-convergence, see Proposition 7.2, p.~68 of \cite{Dalm93}.

\medskip
We pick $u > 0$ and a closed box $B$ (see above (\ref{0.1})). The space $M(B)$ is equipped with the weak topology generated by $C(B)$, as explained in Section 3. We consider $\nu \in M_+(B)$, a finite collection $f_\ell \in C(B)$, $1 \le \ell \le K$, with $f_1 = 1_B$, and a number $\delta > 0$. We define the convex open subset $A$ of $M_+(B)$ consisting of positive measures on $B$ with integrals with respect to $f_\ell$, $1 \le \ell \le K$, $\delta$-close to the corresponding integrals with respect to $\nu$, and we denote by $O$ the homothetic image of $A$ with ratio $u$:
\begin{equation}\label{5.1}
A = \{\rho \in M_+(B); \;|\langle \rho,f_\ell \rangle - \langle \nu, f_\ell\rangle| < \delta, \;\mbox{for} \; 1 \le \ell \le K\}, \;\mbox{and}\;\; O = u A .
\end{equation}

\n
We use the shorthand notation $c(A)$ to denote a positive constant, which depends on $d,B,\nu,K, (f_\ell)_{1 \le \ell \le K}, \delta$. The collection of sets $A$ (or $O$)  above constitutes a base for the relative topology on $M_+(B)$ (viewed as a subset of $M(B)$).

\medskip
We also define for $N \ge 1$, $t \ge 0$ (see (\ref{4.5}) for notation)
\begin{equation}\label{5.2}
\mbox{$\wt{\cL}_{N,t} =$ the restriction of $\cL_{N,t}$ to Borel subsets of $B$}
\end{equation}

\n
(a random element of $M_+(B)$). We now state a consequence of (\ref{4.5}).

\begin{lemma}\label{lem5.1} $(\alpha \ge 0)$
\begin{equation}\label{5.3}
\mbox{As $N \r \infty$, $\wt{\cL}_{N,\alpha}$ converges in distribution to $\cL_{\alpha, B}$ (see above (\ref{3.8}) for notation)}
\end{equation}
\end{lemma}

\begin{proof}
By (\ref{4.5}) and the continuous mapping theorem, see Theorem 5.1, p.\,30 of \cite{Bill68}, it suffices to show that (denoting the set of Radon measures on $\IR^d$ by $\cM_+(\IR^d)$):
\begin{equation}\label{5.4}
\begin{array}{l}
\mbox{the set of continuity points of the map $\rho \in \cM_+(\IR^d) \r \wt{\rho} \in M_+(B)$, where $\wt{\rho}$ is}
\\
\mbox{the restriction of $\rho$ to Borel subsets of $B$, has full measure under the law of $\cL_\alpha$}.
\end{array}
\end{equation}

\n
To see this point, note that by (\ref{1.24}), for a.e.\,$\rho$ under the law of $\cL_\alpha$, one has $\rho (\partial B) = 0$. Hence, for any such $\rho$, for any $V \in C(B)$ (extended as $0$ outside $B$), and any sequence $\rho_n$ converging vaguely to $\rho$, one has $\langle \rho_n, V\rangle \underset{n}{\longrightarrow} \langle \rho, V\rangle$. This proves that $\wt{\rho}_n$ converges weakly to $\wt{\rho}$ in $M_+(B)$ and (\ref{5.4}) follows.
\end{proof}

We then introduce
\begin{equation}\label{5.5}
f_{N,A}(t) = - \log \ov{\IP} \Big[\mbox{\f $\dis\frac{1}{t}$} \; \wt{\cL}_{N,t} \in A\Big], \; \mbox{for} \; t > 0, N \ge 1.
\end{equation}

\n
Since $t \r \wt{\cL}_{N,t}$ has independent, stationary, increments, and $A$ is convex,
\begin{equation}\label{5.6}
\mbox{for each $N \ge 1$, $f_{N,A}(\cdot)$ is subadditive}.
\end{equation}

\n
The next proposition collects some important bounds on $\ov{\IP} [\frac{1}{t} \;\cL_{N,t} \in A]$, which exploit subadditivity. It comes as a step towards Corollary \ref{cor5.3} below. We recall the convention on constants stated below (\ref{5.1}).

\begin{proposition}\label{prop5.2}
When $N \ge c_0(A)$, then for $t_1 \ge c(A)$ and $t \ge 2t_1$, one has
\begin{equation}\label{5.7}
\exp\Big\{ - t \; \dis\frac{f_{N,A}(t_1)}{t_1} - 2t_1 \,c'(A)\Big\} \le \ov{\IP} \Big[ \mbox{\f $\dis\frac{1}{t}$}\;\wt{\cL}_{N,t} \in A\Big] \le \exp\Big\{ - t \;\lim\limits_{s \r \infty} \;\dis\frac{f_{N,A}(s)}{s}\Big\}
\end{equation}

\n
(and the limit in the rightmost term of (\ref{5.7}) exists and is finite).
\end{proposition}

\begin{proof}
We first show that for some $\ve_A = 1/q_A$, where $q_A$ is some positive integer depending on $A$ (with a similar meaning as below (\ref{5.1})) and for some $N_0(A)$
\begin{equation}\label{5.8}
\sup\limits_{N \ge N_0(A)} \;\sup\limits_{t \in [1,1 + \ve_A]} \,f_{N,A}(t) = M < \infty.
\end{equation}

\n
To see this point, we introduce $A'$ defined as $A$ in (\ref{5.1}) with $\delta$ replaced by $\frac{\delta}{2}$. By (\ref{5.3}) and (\ref{3.11}), we know that $\liminf_{N} \ov{\IP}[\wt{\cL}_{N,1} \in A'] \ge \IP[\cL_{1,B} \in A'] > 0$. Hence, for $N \ge N_0(A)$, $\ov{\IP}[\wt{\cL}_{N,1} \in A'] \ge c(A)$. Then, the probability that no trajectory of the interlacement with label in $(d N^{2-d}$, $d(1 + \ve) N^{2-d}]$ enters $NB_N$ is equal to $e^{-d \ve N^{2-d} {\rm cap}_{\IZ^d}(NB_N)} \ge e^{-c(B)\ve}$ (with ${\rm cap}_{\IZ^d}(\cdot)$ the capacity on $\IZ^d$, see for instance (1.57) of \cite{Szni10a}). Such an event is independent under $\ov{\IP}$ of $\{\wt{\cL}_{N,1} \in A'\}$. When $\ve = \ve_A = 1/q_A$ with $q_A$ a large enough integer, so that the set of multiples by a scalar in $[(1 + \ve_A)^{-1},1]$ of a measure in $A'$ is contained in $A$, on the intersection of these two events, one has $\frac{1}{t} \, \wt{\cL}_{N,t} = \frac{1}{t} \,\wt{\cL}_{N,1} \in A$, for all $1 \le t \le 1 + \ve_A$. The claim (\ref{5.8}) follows.

\medskip
We can now apply Lemma 4.2.5, p.~112 of \cite{DeusStro89}, and find that when $t_1 \ge q_A$, then for $t \ge 2t_1$ and $N \ge N_0(A)$, one has (with $[\cdot]$ denoting the integer part)
\begin{equation}\label{5.9}
f_{N,A}(t) \le \Big(\Big[\mbox{\f $\dis\frac{t}{t_1}$}\Big] - 1\Big) \,f_{N,A}(t_1) + 2 t_1 \,M \le t \;\mbox{\f $\dis\frac{f_{N,A}(t_1)}{t_1}$} + 2 t_1 M.
\end{equation}
This proves the first inequality in (\ref{5.7}).

\medskip
As for the  second inequality, we know by the same Lemma 4.2.5 of \cite{DeusStro89}, that the limit in the rightmost term of (\ref{5.7}) exists, is finite, and that moreover, for $t \ge q_A$,
\begin{equation*}
\lim\limits_{s \r \infty} \;\; \dis\frac{f_{N,A}(s)}{s} \le \dis\frac{f_{N,A}(t)}{t}\,.
\end{equation*}
This implies the second inequality in (\ref{5.7}) (and one can choose $c(A) = q_A$ in (\ref{5.7})).
\end{proof}

\n
For fixed $N \ge 1$, the finite dimensional space of signed measures spanned by the basis $\frac{1}{N^d} \; \delta_y, y \in B_N$, is a closed subspace of $M(B)$, and the corresponding coordinates yield an homeomorphism with $\IR^{B_N}$. The intersection of this space with $M_+(B)$ consists of linear combinations of $\frac{1}{N^d} \; \delta_y, y \in B_N$, with non-negative coefficients, and defines a closed convex subset of $M_+(B)$. The finite convex lower semi-continuous  (it is actually continuous) function on $\IR^{B_N}_+$ in (\ref{4.14}), extended to be equal to $+ \infty$ on the complement in $M_+(B)$ of its domain of definition, yields a convex rate function (i.e. a convex, lower semi-continuous, $[0,+\infty]$-valued function) denoted by
\begin{equation}\label{5.10}
\begin{split}
I_N (\mu) & = \wt{\cE}_N (\sqrt{h} - 1, \sqrt{h} - 1), \; \mbox{if $\mu = \mbox{\f $\dis\frac{1}{N^d}$}  \; \dsl_{y \in B_N} h(y)\,\delta_y$, with $h$:  $B_N \r \IR_+$,}
\\ 
& = + \infty, \;\mbox{otherwise}.
\end{split}
\end{equation}
We can now state a consequence of Proposition \ref{prop5.2}. The difference lies in the rightmost inequality, which now involves the functional $I_N$. The next Corollary  \ref{cor5.3} encapsulates the sub-additivity lower and upper bounds, which we will respectively use in the proofs of Theorem \ref{5.4} and of Theorem \ref{theo5.5}.

\begin{corollary}\label{cor5.3} (with $A$ as in (\ref{5.1}))

\medskip
When $N \ge N_0(A)$, for $t_1 \ge c(A)$ and $t \ge 2t_1$, one has
\begin{equation}\label{5.11}
\exp\Big\{ - t \; \dis\frac{f_{N,A}(t_1)}{t_1} - 2t_1 \,c'(A)\Big\} \le \ov{\IP} \Big[ \mbox{\f $\dis\frac{1}{t}$} \;\wt{\cL}_{N,t} \in A\Big] \le \exp\big\{ - t \;\inf\limits_{\rho \in A} \;I_N(\rho)\big\}.
\end{equation}
\end{corollary}

\begin{proof}
Recall that $t \r \wt{\cL}_{N,t}$ has independent, stationary increments. By Theorem 6.1.3, p.~252 of \cite{DembZeit98}, and (\ref{4.6}), (\ref{4.14}), we know that for each fixed $N \ge 1$, the laws of $\frac{1}{n} \;\wt{\cL}_{N,n}$ (on $M_+(B)$), as $n \r \infty$, satisfy a weak large deviation principle at speed $n$, with convex rate function $I_N$, see (\ref{5.10}). Moreover, by the same reference, $\lim_n - \frac{f_{N,A}(n)}{n} = -\inf_A I_N$ ($\in [-\infty,0]$). The comparison with (\ref{5.7}) yields (\ref{5.11}).
\end{proof}

We can now derive the asymptotic lower bound.

\begin{theorem}\label{theo5.4} (Large deviation lower bound, $u > 0$, $O$ as in (\ref{5.1}))
\begin{equation}\label{5.12}
\underset{N}{\underline{\lim}}\; \mbox{\f $\dis\frac{1}{N^{d-2}}$} \; \log \ov{\IP}[\rho_{N,u} \in O] \ge - \inf\limits_{\mu \in O} \; \mbox{\f $\dis\frac{1}{d}$} \;I_u(\mu) \;\mbox{(see (\ref{0.3}) for notation)}.
\end{equation}
\end{theorem}

\begin{proof} We first use sub-additivity. We set $t = \frac{u}{d} \;N^{d-2}$, so that $\rho_{N,u} =  \frac{u}{t} \;\wt{\cL}_{N,t}$ (see (\ref{0.1}) and (\ref{4.5}), (\ref{5.2})), we see by the first inequality of (\ref{5.11}) that for any $t =  \frac{u}{d} \;N^{d-2} \ge 2 t_1 \ge c(A)$ and $N \ge N_0(A)$ (recall $O=uA$), 
\begin{equation*}
\ov{\IP} [\rho_{N,u} \in O] \stackrel{(\ref{5.1})}{=} \ov{\IP}  \Big[\mbox{\f $\dis\frac{1}{t}$} \;\wt{\cL}_{N,t} \in A\Big] \ge \exp\Big\{ -t \; \mbox{\f $\dis\frac{f_{N,A}(t_1)}{t_1}$} - 2t_1 \,c'(A)\Big\}.
\end{equation*}
Hence, for any $t_1 \ge c(A)$,
\begin{equation}\label{5.13}
\underset{N}{\underline{\lim}}\; \mbox{\f $\dis\frac{1}{N^{d-2}}$} \; \log \ov{\IP}[\rho_{N,u} \in O] \ge - \mbox{\f $\dis\frac{u}{d}$}  \;\overline{\lim\limits_N} \;  \mbox{\f $\dis\frac{f_{N,A}(t_1)}{t_1}$} \;.
\end{equation}
We now use convergence in law, and note that by Lemma \ref{lem5.1}
\begin{equation*}
\underset{N}{\underline{\lim}} \; \ov{\IP}[\wt{\cL}_{N,t_1} \in t_1 A] \ge \IP[\cL_{t_1,B} \in t_1 A].
\end{equation*}
Taking logarithms (recall the notation from (\ref{5.5}) and (\ref{3.10})), we find that
\begin{equation*}
- \overline{\lim\limits_N}  f_{N,A}(t_1) \ge - f_{A}(t_1).
\end{equation*}
We are now in position to use the key large deviation result from Section 3. Specifically, coming back to (\ref{5.13}) we see that for $t_1 \ge c(A)$ 
\begin{equation}\label{5.14}
\underset{N}{\underline{\lim}} \;\mbox{\f $\dis\frac{1}{N^{d-2}}$}  \log \ov{\IP} [\rho_{N,u} \in O] \ge - \mbox{\f $\dis\frac{u}{d}$}  \; \mbox{\f $\dis\frac{f_A(t_1)}{t_1}$}  \underset{t_1 \r \infty}{\stackrel{(\ref{3.9})}{\longrightarrow}} -  \mbox{\f $\dis\frac{u}{d}$}  \; \inf\limits_{\rho \in A} I_B(\rho) \stackrel{(\ref{3.5})}{=} - \inf\limits_{\mu \in O} \; \mbox{\f $\dis\frac{1}{d}$}  \;I_u(\mu) \,.
\end{equation}
This proves Theorem \ref{theo5.4}.
\end{proof}

We now turn to the proof of the asymptotic upper bound. An argument based on the convergence of $\Lambda_N(V)$ to $\Lambda(V)$ (or on the asymptotic domination of $\Lambda_N(V)$ by $\Lambda(V)$), for all $V$ in $C(B)$ does not seem straightforward. Instead, we use a strategy, which is in the spirit of $\Gamma$-convergence for the functions $I_N$, see Proposition \ref{prop5.6} below, and Proposition 7.2, p.~68 of \cite{Dalm93}. Another possible, although somewhat indirect, route might be to use the isomorphism theorem of \cite{Szni12b} and the large deviation principles on the empirical distribution functional of the Gaussian free field on $\IZ^d$, proved in \cite{BoltDeus93}.

\begin{theorem}\label{theo5.5} (Large deviation upper bound, $u > 0$, $O$ as in (\ref{5.1}))
\begin{equation}\label{5.15}
\overline{\lim\limits_N} \; \mbox{\f $\dis\frac{1}{N^{d-2}}$}  \; \log \ov{\IP}[\rho_{N,u} \in O] \le - \inf\limits_{\mu \in \ov{O}} \; \mbox{\f $\dis\frac{1}{d}$}  \;I_u(\mu)
\end{equation}

\n
(with $\ov{O}$ the closure of $O$ in $M_+(B)$).
\end{theorem}

\begin{proof} We first exploit sub-additivity. Setting $t = \frac{u}{d} \;N^{d-2}$, we know by Corollary \ref{cor5.3} that for $t \ge 2 c(A)$, $N \ge N_0(A)$ (recall $O=uA$),
\begin{equation*}
\ov{\IP} [\rho_{N,u} \in O] \stackrel{(\ref{5.1})}{=} \ov{\IP}  \Big[ \; \mbox{\f $\dis\frac{1}{t}$}  \;\wt{\cL}_{N,t} \in A\Big] \stackrel{(\ref{5.11})}{\le}  \exp\big\{ -t \;  \inf\limits_{\rho \in A} \; \;I_N(\rho)\big\}.
\end{equation*}
Hence, we find that
\begin{equation}\label{5.16}
\overline{\lim\limits_N} \; \mbox{\f $\dis\frac{1}{N^{d-2}}$}  \; \log \ov{\IP}[\rho_{N,u} \in O] \le - \mbox{\f $\dis\frac{u}{d}$} \; \underset{N}{\underline{\lim}} \; \inf\limits_{\rho \in A} \;I_N(\mu)\,.
\end{equation}
The proof of (\ref{5.15}) relies on the crucial next proposition. 

\begin{proposition}\label{prop5.6} ($K$ compact subset of $M_+(B)$)
\begin{equation}\label{5.17}
\underset{N}{\underline{\lim}} \; \inf\limits_{\rho \in K} \;I_N(\rho) \ge \inf\limits_{\rho \in K} \;I(\rho) \; \mbox{(see (\ref{3.1}), (\ref{5.10}) for notation)}\,.
\end{equation}
\end{proposition}

\bigskip
Let us admit Proposition \ref{prop5.6} for the time being, and first complete the proof of Theorem \ref{theo5.5}. By (\ref{5.16}) and (\ref{5.17}) (with $K = \ov{A}$), we see that
\begin{equation*}
\overline{\lim\limits_N} \; \mbox{\f $\dis\frac{1}{N^{d-2}}$} \;\log \ov{\IP}[\rho_{N,u} \in O] \le - \mbox{\f $\dis\frac{u}{d}$} \; \inf\limits_{\ov{A}} \;I  \underset{(\ref{5.1})}{\stackrel{(\ref{3.5})}{=}} - \inf\limits_{\ov{O}} \; \mbox{\f $\dis\frac{1}{d}$} \;I_u.
\end{equation*}
This proves Theorem \ref{theo5.5} (conditionally on Proposition \ref{prop5.6}).
\end{proof}

There remains to prove Proposition \ref{prop5.6}.

It may be useful at this point to provide an outline of its proof. In essence, assuming that the left-hand side of (\ref{5.17}) is finite (otherwise (\ref{5.17}) is obvious), we will consider a minimizing subsequence $
\mu_\ell$ in $K(\subseteq M_+ (B)$) such that $I_{N_\ell}(\mu_\ell)$ tends to the left-hand side of (\ref{5.17}) and $N_\ell$ tends to infinity. By the relationship between $I_{N_\ell}$ and $\cE_{N_\ell}$ from (\ref{5.10}) and (\ref{4.12}), we will recast $I_{N_\ell}(\mu_\ell)$ as $\cE_{N_\ell}(\varphi_\ell, \varphi_\ell)$, where $\varphi_\ell$ are functions on $\IL_{N_\ell}$, harmonic outside $B_{N_\ell}$, tending to zero at infinity, with value at least $-1$ in $B_{N_\ell}$, and such that $\mu_\ell$ has density $1_{B_{N_\ell}}(1 + \varphi_\ell)^2 N_\ell^{-d}$ with respect to the counting measure on the scaled lattice $\IL_{N_\ell}$. With the help of a cut-off lemma (see Lemma \ref{lem5.7} below), we will replace the sequence $(\varphi_\ell)$ by a sequence $(\overline{\varphi}_\ell)$. In particular, $\overline{\varphi}_\ell$ will coincide with $\varphi_\ell$ on $B_{N_\ell}$, but will vanish outside a fixed compact set, independent of $\ell$, and $\cE_{N_\ell}(\overline{\varphi}_\ell, \overline{\varphi}_\ell)$ will not be substantially bigger than $\cE_{N_\ell}(\varphi_\ell,\varphi_\ell)$. We will then introduce a step function $\Phi_\ell$ on $\IR^d$, constant on cubes of side-length $\frac{1}{N_\ell}$, which coincides with $\overline{\varphi}_\ell$ on $\IL_{N_\ell}$. Making use of the controls on $\overline{\varphi}_\ell$ stemming from the Dirichlet form and the compact support of the functions $\Phi_\ell$, we will show that $\Phi_\ell$, $\ell \ge 1$, is relatively compact in $L^2(\IR^d)$. We will extract a convergent subsequence to a compactly supported $\Phi$ in $L^2(\IR^d)$, having value at least $-1$ on $B$, such that the measure $\mu = (1 + \Phi)^2 m_B$ belongs to $K$ (the compact subset of $M_+(B)$ in the statement of Proposition \ref{prop5.6}), and $I(\mu)$ is not substantially bigger than the left-hand side of (\ref{5.17}). This will show that the infimum of $I$ over $K$ is smaller or equal to $\underline{\lim}_N$ $\underline{\inf}_K I_N$ and conclude the proof of Proposition \ref{prop5.6}.

\bigskip\n
{\it Proof of Proposition \ref{prop5.6}.}
We begin with a cut-off lemma for functions on $\IL_N$, which are harmonic off $B_N$ and tend to $0$ at infinity. This lemma, as mentioned above, will be an important ingredient when proving the relative compactness of a suitable nearly minimizing sequence we later construct.

\medskip
We introduce $L_B \ge 1$ the smallest positive integer such that (see the beginning of Section 1 for notation)
\begin{equation}\label{5.18}
B \subseteq B_\infty(0,L_B),
\end{equation}
and for $R \ge 0$ integer, the closed box
\begin{equation}\label{5.19}
C_R = B_\infty \big(0,L_B(R + 2)\big) \supseteq B,
\end{equation}

\n
so that $d_\infty\big(B_\infty(0,L_B),  \partial C_R\big) \ge L_B(R+1)$. For functions $\varphi$ defined on $\IL_N$, we use, as a shorthand, the notation $\sup_B \varphi$, $\inf_{C_R} \varphi, \dots,$ in place of $\sup_{B \cap \IL_N} \varphi$, $\inf_{C_R \cap \IL_N} \varphi, \dots$~.

\medskip
We are now ready to state and prove the cut-off lemma. The reader may choose to first skip its proof.

\begin{lemma}\label{lem5.7} (Cut-off lemma, $R \ge 1$ integer)

\medskip
Let $N \ge 1$, and $\varphi$ on $\IL_N$ be harmonic outside $B_N$ and tend to $0$ at infinity. There exists $\ov{\varphi}$ on $\IL_N$ such that
\begin{equation}\label{5.20}
\hspace{-4cm} \left\{ \begin{array}{rl}
{\rm i)}  &\mbox{$\ov{\varphi} = \varphi$, on $C_R$,}
\\[1ex]
{\rm ii)} & \mbox{$\ov{\varphi} = 0$ outside $C_{100R}$,}
\\[1ex]
{\rm iii)} & \cE_N(\ov{\varphi},\ov{\varphi}) \le \cE_N(\varphi,\varphi)\Big(1 +  \mbox{\f $\dis\frac{c(B)}{R^{d-2}}$}\Big).
\end{array}\right.
\end{equation}
\end{lemma}

\begin{proof}
To simplify notation, all constants in the proof implicitly depend on $d$ and $B$.

\medskip
We let $\psi$ and $\gamma$ stand for the harmonic extensions of $\varphi_+ = \max\{\varphi,0\}$ and $\varphi_- = \max\{-\varphi,0\}$ outside $B_N$ (on $\IL_N \backslash B_N)$, which tend to $0$ at infinity. Since $\wt{\cE}_N$ in (\ref{4.11}) is a Dirichlet form, we know that the restrictions $\wt{\varphi}, \wt{\varphi}_+, \wt{\varphi}_-$ of $\varphi, \varphi_+, \varphi_-$ to $B_N$ satisfy 
\begin{equation}\label{5.21}
\begin{split}
\cE_N(\psi,\psi)  & \stackrel{(\ref{4.12})}{=} \wt{\cE}_N(\wt{\varphi}_+,\wt{\varphi}_+) \le \wt{\cE}_N (\wt{\varphi},\wt{\varphi}) \stackrel{(\ref{4.12})}{=} \cE_N (\varphi,\varphi) \; \mbox{and}
\\[1ex]
\cE_N(\gamma,\gamma) &\;\; \le \cE_N(\varphi,\varphi) \; \mbox{(in a similar fashion)}.
\end{split}
\end{equation}

\n
By the Harnack inequality and chaining (see Theorem 1.7.2, p.~42 of \cite{Lawl91}), we have
\begin{equation}\label{5.22}
\max\limits_{ \partial C_0} \psi \le c \;\min\limits_{ \partial C_0} \psi \; \mbox{and} \; \max\limits_{ \partial C_0} \gamma \le c \;\min\limits_{ \partial C_0} \gamma
\end{equation}

\n
(we recall that all constants in the proof of Lemma \ref{lem5.7} implicitly depend on $d$ and $B$). Since $\psi$ is harmonic outside $C_0$ and tends to zero at infinity, it follows, by a stopping argument, that for $y$ in $ \partial C_R$, $\psi(y)$ is smaller than the product of $\max_{ \partial C_0} \psi$ with the probability for the walk on $\IL_N$ starting at $y$ to reach $ \partial C_0$. A similar bound holds for $\gamma$. By (\ref{5.22}) and classical random walk estimates, we obtain
\begin{equation}\label{5.23}
\max\limits_{ \partial C_R} \psi \le \mbox{\f $\dis\frac{c}{R^{d-2}}$} \;\min\limits_{ \partial C_0} \psi \; \stackrel{\rm def}{=} a_R, \; \max\limits_{ \partial C_R} \gamma \le  \mbox{\f $\dis\frac{c}{R^{d-2}}$} \;\min\limits_{ \partial C_0} \gamma \stackrel{\rm def}{=} b_R, \; \mbox{for $R \ge 1$}.
\end{equation}

\n
By estimates on the discrete gradient of harmonic functions in large balls, see Theorem 1.7.1, p.~42 of \cite{Lawl91}, we have for $R \ge 1$, $y \in \IL_N \backslash C_R$ and $y' \sim y$ in $\IL_N$,
\begin{equation}\label{5.24}
|\psi(y') - \psi(y)| \le \mbox{\f $\dis\frac{c}{R^{d-2}}$} \; \mbox{\f $\dis\frac{1}{RN}$} \; \min\limits_{ \partial C_0} \psi, \; |\gamma(y') - \gamma(y)| \le \mbox{\f $\dis\frac{c}{R^{d-2}}$} \; \mbox{\f $\dis\frac{1}{RN}$}\; \min\limits_{ \partial C_0} \gamma.
\end{equation}
With a similar notation as below (\ref{4.11}), we define for $R \ge 1$
\begin{equation*}
q_R(y) = P^N_y [H_{C_R \cap \IL_N} > T_{C_{100R} \cap \IL_N}], \; \mbox{for $y \in \IL_N$},
\end{equation*}

\n
the probability for the walk on $\IL_N$ starting at $y$ to exit $C_{100R}$ before entering $C_R$. We then define
\begin{equation}\label{5.25}
\begin{split}
\ov{\varphi}(y) & = \varphi(y) \; \mbox{on} \; C_R \cap \IL_N,
\\[1ex]
& = \psi(y) \wedge \big(a_R \,q_R(y)\big) -  \gamma(y) \wedge \big(b_R \,q_R(y)\big), \;\mbox{on} \; \IL_N \backslash C_R
\end{split}
\end{equation}

\n
(note that the expressions on both lines are equal for $y \in  \partial C_R \cap \IL_N$).

\medskip
For simple random walk on $\IZ^d$, the capacity of $(NC_R) \cap \IZ^d $ relative to  $(NC_{100R}) \cap \IZ^d $ is at most $c(NR)^{d-2}$, see for instance (\ref{5.26}) of \cite{Szni09b}. Moreover, it is equal to the Dirichlet form of the function $q_R(\frac{\cdot}{N})$. One thus has the bound
\begin{equation}\label{5.26}
\cE_N(q_R,q_R) \le c\,R^{d-2}.
\end{equation}

\n
Further, note that when $y' \sim y$ in $\IL_N$ (see below (\ref{4.9}) for notation) are not both in $C_R$,
\begin{equation*}
|\ov{\varphi}(y') - \ov{\varphi}(y) |^2  \le 2 (| \psi (y') - \psi(y)|^2 + a^2_R |q_R(y') - q_R(y)|^2 + |\gamma(y') - \gamma(y)|^2 + b^2_R |q_R (y') - q_R(y)|^2).
\end{equation*}

\n
Thus, coming back to (\ref{5.25}), we see that, denoting by $\sum '$ the summation over $y' \sim y$ in $\IL_N$, not both in $C_R$, one has
\begin{equation}\label{5.27}
\begin{split}
\cE_N(\ov{\varphi},\ov{\varphi})&  \le\; \cE_N (\varphi,\varphi) + \mbox{\f $\dis\frac{c}{N^{d-2}}$} \;\dsl \,\!' \big(|\psi(y') - \psi(y)|^2 + | \gamma(y') - \gamma(y)|^2\big)
\\
&\quad \; + 2 (a^2_R + b^2_R) \,\cE_N(q_R,q_R)
\\[1ex]
&\!\!\!\! \stackrel{(\ref{5.24})}{\le}\; \cE_N(\varphi,\varphi) + \Big(\mbox{\f $\dis\frac{c}{N^{d-2}}$} \;\dsl_{\ell \ge NR} \; \mbox{\f $\dis\frac{\ell^{d-1}}{(\frac{\ell}{N})^{2(d-2)}}$} \; \mbox{\f $\dis\frac{1}{\ell^2}$} + \mbox{\f $\dis\frac{c}{R^{d-2}}$} \Big) \big(\min\limits_{ \partial C_0} \psi^2 + \min\limits_{ \partial C_0} \gamma^2\big)
\\[1ex]
&\le \; \cE_N(\varphi,\varphi)  + \mbox{\f $\dis\frac{c}{R^{d-2}}$}\big(\min\limits_{ \partial C_0} \psi^2 + \min\limits_{ \partial C_0} \gamma^2\big).
\end{split}
\end{equation}

\n
Since the capacity of a finite subset of $\IZ^d$ is smaller than the Dirichlet form of any function in the extended Dirichlet space equal to 1 on the set, and $\psi$, resp.~$\gamma$, is bigger or equal to $\min_{ \partial C_0} \psi$, resp.~$\min_{ \partial C_0} \gamma$, on $ (\partial C_0) \cap \IL_N$, we find that
\begin{equation*}
\cE_N(\psi,\psi) \ge \mbox{\f $\dis\frac{d}{N^{d-2}}$} \;{\rm cap}_{\IZ^d} (N \partial C_0) (\min\limits_{ \partial C_0} \psi)^2 \ge c(\min\limits_{ \partial C_0} \psi)^2 \;\; \mbox{and likewise} \;\; \cE_N (\gamma,\gamma) \ge c(\min\limits_{ \partial C_0} \gamma)^2.
\end{equation*}

\n
Inserting these bounds in the last line of (\ref{5.27}) and using (\ref{5.21}), we find
\begin{equation*}
\cE_N(\ov{\varphi},\ov{\varphi}) \le \cE_N(\varphi,\varphi) \Big(1 +  \mbox{\f $\dis\frac{c}{R^{d-2}}$}\Big).
\end{equation*}

\n
Moreover, by (\ref{5.25}), $\ov{\varphi} = \varphi$ on $C_R$ and $0$ on $\IL_N \backslash C_{100R}$. We have proved Lemma \ref{lem5.7}. \hfill $\square$

\bigskip
We resume the proof of (\ref{5.17}). We denote by $\alpha \in [0,\infty]$ the left-hand side of (\ref{5.17}). Without loss of generality, we can assume $\alpha < \infty$, otherwise (\ref{5.17}) is immediate. We consider a subsequence $N_\ell$, $\ell \ge 1$, as well as sequences $\mu_\ell \in M_+(B)$, and $\varphi_\ell \ge -1$ on $B_{N_\ell}$, of functions on $\IL_{N_\ell}$ harmonic outside $B_{N_\ell}$ and tending to $0$ at infinity, such that using (\ref{5.10}) and (\ref{4.12})
\begin{equation}\label{5.28}
\mu_\ell \in K, \; \mbox{and} \; I_{N_\ell}(\mu_\ell) \r \alpha ~~(\stackrel{\rm def}{=} \underset{N}{\underline{\lim}} \; \inf\limits_{\rho \in K} \;I_N(\rho)),
\end{equation}
\begin{equation}\label{5.29}
\left\{ \begin{array}{rl}
{\rm i)}   & \quad \mu_\ell = \mbox{\f $\dis\frac{1}{N^d_\ell}$} \; \dsl_{y \in B_{N_\ell}} \big(1 + \varphi_\ell(y)\big)^2 \delta_y, 
\\[2ex]
{\rm ii)}  &\quad  I_{N_\ell}(\mu_\ell) = \cE_{N_\ell}(\varphi_\ell, \varphi_\ell).
\end{array}\right.
\end{equation}

\n
We choose $R \ge 1$, and construct with the help of Lemma \ref{lem5.7} a sequence $\ov{\varphi}_\ell, \ell \ge 1$, of functions on $\IL_{N_\ell}$ for which (\ref{5.20}) holds (with $\varphi$ replaced by $\varphi_\ell$ and $N$ by $N_\ell$). The functions $\ov{\varphi}_\ell$ vanish outside $C_{100R} \cap \IL_{N_\ell}$. Since the principal Dirichlet eigenvalue of the discrete Laplacian in a box of side-length $L \ge 1$ in $\IZ^d$ is at least $c \,L^{-2}$, see for instance \cite{GlimJaff81}, p.~185, we find that
\begin{equation}\label{5.30}
\|\ov{\varphi}_\ell\|_{L^2(\IL_{N_\ell})} \le c(B) \,R^2 \,\cE_{N_\ell} (\ov{\varphi}_\ell,\ov{\varphi}_\ell) \; \mbox{(and the limsup in $\ell$ is at most $c'(B) \,R^2 \alpha$)}.
\end{equation}

\n
The functions $\ov{\varphi}_\ell$ are defined on the different lattices $\IL_{N_\ell}$, and it is convenient to introduce the functions $\Phi_\ell$, $\ell \ge 1$, on $\IR^d$, which take the value $\ov{\varphi}_\ell(y)$ on $y + \frac{1}{N^\ell}\;[0,1)^d$, i.e.
\begin{equation}\label{5.31}
\Phi_\ell(z) = \dsl_{y \in \IL_{N_\ell}} \ov{\varphi}_\ell(y) \;1_{y + \frac{1}{N_\ell}[0,1)^d}(z), \; z \in \IR^d, \,\ell \ge 1.
\end{equation}
Note that, by construction, for all $\ell \ge 1$,
\begin{align}
&1 + \Phi_\ell(z) \ge 0, \; \mbox{if $z \in B$ and $d_\infty(z,\partial B) \ge \mbox{\f $\dis\frac{1}{N_\ell}$}$}, \label{5.32}
\\[2ex]
& \Phi_\ell = 0, \; \mbox{on} \; \IR^d \backslash C_{200R}, \label{5.33}
\\[2ex]
&\|\Phi_\ell\|_{L^2(\IR^d)} = \|\ov{\varphi}_\ell\|_{L^2(\IL_{N_\ell})} \; \mbox{(and $\sup\limits_\ell \|\Phi_\ell\|_{L^2(\IR^d)} < \infty$, by (\ref{5.30}))}. \label{5.34}
\end{align}

\n
We will now prove that the functions $\Phi_\ell$, $\ell \ge 1$, are equicontinuous in $L^2(\IR^d)$ with respect to translations:
\begin{equation}\label{5.35}
\lim\limits_{h \r 0} \;\sup\limits_{\ell \ge 1} \| \Phi_\ell(\cdot +h) - \Phi_\ell(\cdot)\|_{L^2(\IR^d)} = 0.
\end{equation}

\n
Using the triangle inequality and translation invariance, we can assume, without loss of generality, that $h$ is parallel to and pointing in the direction of $e_i$, the $i$-th vector of the canonical basis of $\IR^d$. For notational simplicity, we treat the case $i=1$ (the other cases are handled similarly).

\medskip
We write $N_\ell h = (k+r) e_1$, where $k \ge 0$, is an integer and $0 \le r < 1$ (both depend on $\ell$). We see that
\begin{equation*}
\begin{array}{l}
\|\Phi_\ell (\cdot + h) - \Phi_\ell (\cdot)\|^2_{L^2(\IR^d)} \le 2(a+b), \; \mbox{where}
\\[1ex]
a = \Big\| \Phi_\ell \Big( \cdot + \mbox{\f $\dis\frac{k}{N_\ell}$} \;e_1\Big) - \Phi_\ell(\cdot)\Big\|^2_{L^2(\IR^d)} \; \mbox{and} \; b = \Big\| \Phi_\ell \Big( \cdot + \mbox{\f $\dis\frac{r}{N_\ell}$} \;e_1\Big) - \Phi_\ell(\cdot)\Big\|^2_{L^2(\IR^d)}.
\end{array}
\end{equation*}

\n
Expressing $b$ in terms of $\ov{\varphi}_\ell$, we find that
\begin{equation*}
b = r \Big\| \ov{\varphi}_\ell \Big( \cdot + \mbox{\f $\dis\frac{e_1}{N_\ell}$}\Big) - \ov{\varphi}_\ell(\cdot)\Big\|^2_{L^2(\IL_{N_\ell})} \stackrel{(\ref{4.9})}{\le} c \; \mbox{\f $\dis\frac{r}{N^2_\ell}$} \; \cE_{N_\ell} (\ov{\varphi}_\ell, \ov{\varphi}_\ell) \le \mbox{\f $\dis\frac{c}{N^2_\ell}$} \;|h| \,\cE_{N_\ell}  (\ov{\varphi}_\ell, \ov{\varphi}_\ell).
\end{equation*}

\n
On the other hand, by the  triangle inequality and translation invariance, we have
\begin{equation*}
\begin{array}{l}
a \le k^2 \,\|\Phi_\ell  \Big( \cdot + \mbox{\f $\dis\frac{e_1}{N_\ell}$}\Big) - \Phi_\ell (\cdot)\Big\|^2_{L^2(\IR^d)} = k^2 \| \ov{\varphi}_\ell  \Big( \cdot + \mbox{\f $\dis\frac{e_1}{N_\ell}$}\Big)  - \ov{\varphi}_\ell (\cdot)\Big\|^2_{L^2(\IL_{N_\ell})}
\\[2ex]
\stackrel{(\ref{4.10})}{\le} c\;\mbox{\f $\dis\frac{k^2}{N^2_\ell}$}\;\cE_{N_\ell} (\ov{\varphi}_\ell, \ov{\varphi}_\ell) \le c\,|h|^2 \,\cE_{N_\ell} (\ov{\varphi}_\ell, \ov{\varphi}_\ell).
\end{array}
\end{equation*}

\n
Combining the bounds on $a$ and $b$, we see that for $h = |h|\,e_1$,
\begin{equation}\label{5.36}
\|\Phi_\ell(\cdot + h) - \Phi_\ell (\cdot) \|^2_{L^2(\IR^d)} \le c\,|h| \Big(|h| \vee \;\mbox{\f $\dis\frac{1}{N_\ell}$}\Big) \,\cE_{N_\ell}(\ov{\varphi}_\ell,\ov{\varphi}_\ell), \;\ \mbox{for} \; \ell \ge 1 .
\end{equation}

\n
The claim (\ref{5.35}) now follows since $\sup_\ell \cE_{N_\ell} (\ov{\varphi}_\ell,\ov{\varphi}_\ell) < \infty$.

\medskip
By (\ref{5.30}), (\ref{5.33}), (\ref{5.35}), and Theorem 2.21, p.~31 of \cite{Adam75}, we find that $\Phi_\ell$, $\ell \ge 1$, is a relatively compact subset of $L^2(\IR^d)$. 

\medskip
Hence, up to extraction of a subsequence (which we still denote by $\Phi_\ell$), we can assume that
\begin{equation}\label{5.37}
\Phi_\ell \underset{\ell}{\longrightarrow} \Phi \; \mbox{in} \;L^2(\IR^d),
\end{equation}
and by (\ref{5.32}), (\ref{5.33}), we see that
\begin{equation}\label{5.38}
\mbox{$1 + \Phi \ge 0$ ~a.e.\,on $B$ and $\Phi = 0$ ~a.e.\,on $\IR^d \backslash C_{200R}$}.
\end{equation}
By (\ref{5.37}), we see that
\begin{equation}\label{5.39}
(1 + \Phi_\ell)^2 \underset{\ell}{\longrightarrow} (1 + \Phi)^2  \; \mbox{in} \; L^1_{\rm loc}(\IR^d) .
\end{equation}

\n
By uniform integrability, it follows that the integral of $(1 + \Phi_\ell)^2$ over $\{z \in \IR^d$; $d_\infty(z,\partial B) \le \frac{1}{N_\ell}\}$ tends to $0$ with $\ell$. In addition, since $\ov{\varphi}_\ell = \varphi_\ell$ on $C_R \supseteq \{z \in \IR^d$; $d_\infty(z,B) \le 1\}$, cf.~(\ref{5.19}), we see by (\ref{5.29}) i) and (\ref{5.31}) that $\mu_\ell (B \cap \{z \in \IR^d$; $d_\infty(z,\partial B) \le \frac{1}{N_\ell}\}) \underset{\ell}{\longrightarrow} 0$. Thus, letting $B'_N$ stand for the $y \in \IL_N$ such that $y + \frac{1}{N} \,[0,1)^d \subseteq B$, we find that for $V \in C(B)$,
\begin{equation*}
\langle \mu_\ell,V\rangle - \dis\int_B (1 + \Phi_\ell)^2 V\,dz = \dsl_{y \in B'_{N_\ell}} \dis\int_{y + \frac{1}{N_\ell}[0,1)^d}\big(V(y) - V(z)\big) \big(1 + \Phi_\ell(z)\big)^2 dz + o(1), \;\mbox{as} \; \ell \r \infty.
\end{equation*}

\n
Moreover, the sum in the right-hand side tends to $0$ with $\ell$, by uniform continuity of $V$ and (\ref{5.39}). Hence, by (\ref{5.39}), we see that $\mu_\ell$ ($\in K$, by (\ref{5.28})) converges in $M_+(B)$ to $(1 + \Phi)^2 m_B$. This shows that
\begin{equation}\label{5.40}
\mu = (1 + \Phi)^2 m_B \in K.
\end{equation}

\medskip\n
Moreover, by (\ref{5.20}) iii), (\ref{5.28}), (\ref{5.29}) ii), we find that
\begin{equation}\label{5.41}
\alpha \Big( 1 + \mbox{\f $\dis\frac{c(B)}{R^{d-2}}$}\Big) \ge \overline{\lim\limits_\ell} \;\cE_{N_\ell} (\ov{\varphi}_\ell,\ov{\varphi}_\ell) = \overline{\lim\limits_\ell} \; \fr \; \dsl^d_{j=1} N^2_\ell \Big\| \;\ov{\varphi}_\ell \Big( \cdot + \mbox{\f $\dis\frac{e_j}{N_\ell}$}\Big) - \ov{\varphi}_\ell(\cdot)\Big\|^2_{L^2(\IL_{N_\ell})}
\end{equation}

\n
and expressing this last quantity in terms of $\Phi_\ell$ and using the Fourier transform
\begin{align*}
&=   \overline{\lim\limits_\ell} \; \fr \; \dsl^d_{j=1}  \mbox{\f $\dis\frac{1}{(2 \pi)^d}$} \;\dis\int_{\IR^d} N^2_\ell \; \Big|e^{i \frac{\xi_j}{N_ \ell}} - 1\Big|^2 \,|\wh{\Phi}_\ell(\xi)|^2 d\xi
\\[1ex]
&\!\! \stackrel{\rm Fatou}{\ge}  \mbox{\f $\dis\frac{1}{2(2 \pi)^d}$} \;\dis\int_{\IR^d} |\xi|^2 \,|\wh{\Phi}(\xi)|^2 d\xi = \fr \; \dis\int_{\IR^d} |\nabla \Phi (z)|^2 dz
\end{align*}

\n
(extracting some subsequence along which $\wh{\Phi}_\ell$ converges a.e.\,to $\wh{\Phi}$, before the last inequality). We have thus shown that $\Phi \in H^1(\IR^d)$ and
\begin{equation}\label{5.42}
\cE (\Phi,\Phi) \le \alpha \Big(1 + \mbox{\f $\dis\frac{c(B)}{R^{d-2}}$}\Big) .
\end{equation}

\n
Combined with (\ref{5.38}), (\ref{5.40}), this implies, see (\ref{3.4}), that
\begin{equation}\label{5.43}
\inf\limits_K \;I \le \alpha \Big(1 +  \mbox{\f $\dis\frac{c(B)}{R^{d-2}}$}\Big), \;\mbox{for any $R > 1$}.
\end{equation}
Letting $R \r \infty$, we obtain Proposition \ref{prop5.6}.
\end{proof}

We now come to the main theorem of this section, see (\ref{0.1}), (\ref{0.3}) for notation.

\begin{theorem}\label{theo5.8} (The large deviation principle, $u > 0$)
\begin{equation}\label{5.44}
\begin{array}{l}
\mbox{As $N \r \infty$, the laws of $\rho_{N,u}$ on $M_+(B)$ satisfy a large deviation principle}
\\
\mbox{at speed $N^{d-2}$, with convex, good rate function $\frac{1}{d} \;I_u$}.
\end{array}
\end{equation}
\end{theorem}

\begin{proof}
Combining Theorems \ref{theo5.4} and \ref{theo5.5}, it follows from Theorem 4.1.11, p.~120 of \cite{DembZeit98}, that $\rho_{N,u}$ satisfies a weak large deviation principle at speed $N^{d-2}$, with rate function $\frac{1}{d} \,I_u$ (which is good and convex by (\ref{3.2}), (\ref{3.5})).

\medskip
In addition, one has exponential tightness for the laws of the $\rho_{N,u}$, due to the fact that (see (\ref{3.3}) and (\ref{3.11}) of \cite{Szni13a}) for small $\lambda > 0$,
\begin{align*}
\overline{\lim\limits_N} & \;\mbox{\f $\dis\frac{1}{N^{d-2}}$} \;\log \ov{\IE} \big[ \exp\big
\{N^{d-2} \langle \rho_{N,u}, \lambda 1_B\rangle\big\}\big]  \underset{(\ref{4.7})}{\stackrel{(\ref{4.6})}{=}}
\\ 
\overline{\lim\limits_N} & \; \mbox{\f $\dis\frac{1}{d}$} \;\log \ov{\IE} [e^{\langle \cL_{N,u},d\lambda 1_B\rangle}] = u\,c(B,\lambda) < \infty,
\end{align*}

\n
and the Chebyshev Inequality. The claim (\ref{5.44}) now follows (see~p.~8 of \cite{DembZeit98}).
\end{proof}

\section{An application}
\setcounter{equation}{0}

In this section, we apply the large deviation principle proved in the last section, see Theorem \ref{theo5.8}, and control the probability of existence of ``high local density'' regions insulating a given compact subset of $\IR^d$. Our main results appear in Theorems \ref{theo6.2} and \ref{theo6.4}. Extensions are discussed in Remark \ref{rem6.5}. We begin with some definitions and preliminary remarks.

\medskip
We consider a compact set $K$ and a closed box $B_0$ in $\IR^d$ such that
\begin{equation}\label{6.1}
\phi \not= K \subset B_0 .
\end{equation}

\medskip\n
Given $a \in \IR$ and a continuous function $f$ on $\IR^d$, we say that $\{f \ge a\}$ {\it disconnects $K$ from $\partial B_0$}, if
\begin{equation}\label{6.2}
\begin{array}{l}
\mbox{for any continuous function $\psi$: $[0,1] \r B_0$, such that $\psi(0) \in K$ and}
\\
\mbox{$\psi(1) \in \partial B_0$, one has $\sup\limits_{0 \le t \le 1} f\big(\psi(t)\big) \ge a$.}
\end{array}
\end{equation}

\n
Note that the collection of bounded continuous functions $f$ for which $\{f \ge a\}$ disconnects $K$ from $\partial B_0$ is closed for the sup-norm topology.

\medskip
Further, we consider $\delta \in (0,1)$, and a closed box $B$, so that
\begin{equation}\label{6.3}
K \subset B_0 \subset B, \; \mbox{and $d(\partial B_0, \partial B) > \delta$},
\end{equation}

\n
as well as a continuous probability density $\varphi_\delta$ (with respect to Lebesgue measure), which is supported in $B(0,\delta)$. Denoting by $C_0(\IR^d)$ the set of continuous functions on $\IR^d$ that tend to $0$ at infinity (endowed with the sup-norm topology), we consider the regularization map $r_\delta$ from $M_+ (B)$ into $C_0(\IR^d)$
\begin{equation}\label{6.4}
\mu \in M_+(B) \r r_\delta(\mu)(\cdot) = \dis\int_B \varphi_\delta (\cdot - y) \,\mu(dy) \in C_0(\IR^d),
\end{equation}
and introduce for $a \ge 0$, $\delta \in (0,1)$, the subset of $M_+(B)$
\begin{equation}\label{6.5}
\cD_{a,\delta} = \{ \mu \in M_+(B); \;\{r_\delta(\mu) \ge a\} \;\mbox{disconnects $K$ from $\partial B_0\}$}.
\end{equation}

\n
The next lemma collects some useful properties of the above objects.

\begin{lemma}\label{lem6.1} $(a \ge 0$, $0 < \delta < 1$, $u > 0)$
\begin{align}
& \mbox{$r_\delta$ is continuous}. \label{6.6}
\\[1ex]
&\mbox{$\cD_{a,\delta}$ is a closed subset of $M_+(B)$}.\label{6.7}
\\[1ex]
&\mbox{The event $\{\rho_{N,u} \in \cD_{a,\delta}\}$ does not depend on the choice of $B$ (satisfying (\ref{6.3}))}. \label{6.8}
\end{align}
\end{lemma}

\begin{proof}
We start with the proof of (\ref{6.6}). The functions $\varphi_\delta(z - \cdot)$ on $B$, as $z$ varies in $\IR^d$, are equicontinuous at each point of $B$ and uniformly bounded. It now follows from Theorem 6.8, p.~51 of \cite{Part67}, that when $\mu_n$ converges weakly to $\mu$ in $M_+(B)$, $r_\delta(\mu_n)$ converges uniformly to $r_\delta(\mu)$. This proves (\ref{6.6}). Then, (\ref{6.7}) is an immediate consequence of (\ref{6.6}) and the observation below (\ref{6.2}). As for (\ref{6.8}), it suffices to notice that the restriction of $r_\delta(\rho_{N,u})$ to $B_0$ does not depend on the choice of $B$ when (\ref{6.3}) holds.
\end{proof}

By the observation below (\ref{4.4}) and (\ref{6.6}), we know that for $u > 0$, $\ov{\IP}$-a.s., $r_\delta(\rho_{N,u})$ converges uniformly to $r_\delta(um_B) (\cdot) = u \int_B \varphi_\delta(\cdot - y)dy$, and this function equals $u$ on $B_0$ by (\ref{6.3}). We will now consider the case where $a > u$ and study the large $N$ behavior of the probability of occurence of high values of $r_\delta(\rho_{N,u})$ insulating $K$ from $\partial B_0$.

\begin{theorem}\label{theo6.2} (Insulation upper bound, $a > u > 0$)
\begin{equation}\label{6.9}
\overline{\lim\limits_N} \; \mbox{\f $\dis\frac{1}{N^{d-2}}$} \; \log \ov{\IP} [\rho_{N,u} \in \cD_{a,\delta}] \le - \mbox{\f $\dis\frac{1}{d}$}  \;\big(\sqrt{a} - \sqrt{u}\big)^2 {\rm cap}(K)
\end{equation}
\end{theorem}

\begin{proof}
By (\ref{6.7}) and Theorem \ref{theo5.8} we have
\begin{equation}\label{6.10}
\overline{\lim\limits_N} \; \mbox{\f $\dis\frac{1}{N^{d-2}}$} \; \log \ov{\IP} [\rho_{N,u} \in \cD_{a,\delta}] \le  - \inf\limits_{\cD_{a,\delta}} \;  \mbox{\f $\dis\frac{1}{d}$}  \;I_u.
\end{equation}

\n
We will use the notation $I_{u,B_0}$ or $I_{u,B}$ to highlight the dependence on the closed box at hand in the definition (\ref{0.3}). We have the following control.
\begin{lemma}\label{lem6.3}
\begin{equation}\label{6.11}
I_{u,B_0} \big(r_\delta(\mu) m_{B_0}\big) \le I_{u,B}(\mu), \; \mbox{for any $\mu \in M_+(B)$}
\end{equation}

\n
(we view $r_\delta(\mu) \,m_{B_0}$ as an element of $M_+(B_0)$).
\end{lemma}

\begin{proof}
Without loss of generality, we can assume that $I_{u,B}(\mu) < \infty$, so that $\mu = hm_B$ with $h \in L^1_+(m_B)$. Note that $y \in \IR^d \r 1_{B_0}(\cdot) \,h(\cdot - y) \in L^1(m_{B_0})$ is a continuous map (we extend $h$ outside $B$ as being equal to $0$). Moreover, by (\ref{6.4}) 
\begin{equation*}
\big(1_{B_0} r_\delta(\mu)\big) (\cdot) = \dis\int 1_{B_0}(\cdot) h(\cdot - y) \,\varphi_\delta(y)dy.
\end{equation*}

\medskip
We also know that $f \in L^1_+(m_{B_0}) \r I_{u,B_0}(fm_{B_0})$ is a convex, lower semi-continuous map, see (\ref{2.32}), (\ref{3.1}), (\ref{3.5}). Hence, we have
\begin{equation}\label{6.12}
\begin{split}
I_{u,B_0} (r_\delta(\mu) m_{B_0})  \le&\; \sup\limits_{|y| \le \delta} \;I_{u,B_0} (1_{B_0} (\cdot) \,h(\cdot -y) m_{B_0})
\\[1ex]
\underset{(\ref{3.5})}{\stackrel{(\ref{3.1})}{=}} &\; \sup\limits_{|y| \le \delta}  \;I_{u,B_0 -y} (1_{B_0-y}\, hm_{B_0-y})
\\[1ex]
\underset{(\ref{6.3})}{\stackrel{(\ref{3.3})}{\le}} &\;I_{u,B} (h m_B) = I_{u,B}(\mu).
\end{split}
\end{equation}
This proves (\ref{6.11}).
\end{proof}

We will now bound the right-hand side of (\ref{6.10}). Given $\mu \in \cD_{a,\delta}$, we define
\begin{equation}\label{6.13}
\varphi(z) = E_z \big[\big(\sqrt{r_\delta(\mu)} - \sqrt{u}\big)(X_{H_{B_0}}), \;H_{B_0} < \infty], \; z \in \IR^d,
\end{equation}

\n
so that $\varphi \in C_0(\IR^d)$. If $I_{u,B_0}(r_\delta(\mu) m_{B_0}) < \infty$, then by (\ref{3.1}), (\ref{2.32}), and the explanation below (\ref{2.27}), we have $\varphi \in \cF_e$ and
\begin{equation}\label{6.14}
\cE(\varphi,\varphi) = \wt{\cE}_{B_0} \big(\sqrt{r_\delta(\mu)} - \sqrt{u}, \sqrt{r_\delta(\mu)}  - \sqrt{u}\big) \underset{(\ref{3.5}),(\ref{2.32})}{\stackrel{(\ref{3.1})}{=}} I_{u,B_0} (r_\delta(\mu) m_{B_0}).
\end{equation}

\medskip\n
In addition, $\varphi = \sqrt{r_\delta(\mu)} - \sqrt{u}$ on $B_0$, and since $\mu \in \cD_{a,\delta}$, we see that $\{\varphi \ge \sqrt{a} - \sqrt{u}\}$ disconnects $K$ from $\partial B_0$. Setting $\wt{K} = \{y \in B_0$,  $\varphi(y) \ge \sqrt{a} - \sqrt{u}\}$, it follows from Theorem 1.10, p.~58 of \cite{PortSton78}, that ${\rm cap}(\wt{K}) \ge {\rm cap}(K)$, and from p.~71 of \cite{FukuOshiTake94}, that $\cE(\varphi,\varphi)  \ge \big(\sqrt{a} - \sqrt{u}\big)^2 {\rm cap} (\wt{K})$. Hence, for $\mu \in \cD_{a,\delta}$, we have
\begin{equation}\label{6.15}
I_u(\mu) \stackrel{(\ref{6.11}), (\ref{6.14})}{\ge} \cE (\varphi,\varphi)  \ge \big(\sqrt{a} - \sqrt{u}\big)^2 {\rm cap} (\wt{K}) \ge \big(\sqrt{a} - \sqrt{u}\big)^2 {\rm cap}(K).
\end{equation}

\n
Inserting this bound in the right-hand side of (\ref{6.10}), we obtain (\ref{6.9}).
\end{proof}

We now complement the asymptotic upper bound from Theorem \ref{theo6.2} with an asymptotic lower bound. We denote by $K^\delta = \{z \in \IR^d$; $d(z,K) \le \delta\}$ the closed $\delta$-neighborhood of $K$ for the Euclidean distance. One knows that ${\rm cap}(K^\delta) \downarrow {\rm cap}(K)$ as $\delta \r 0$, see Remark \ref{rem6.5} 1) below.

\begin{theorem}\label{theo6.4} (Insulation lower bound, $a > u > 0$)
\begin{equation}\label{6.16}
\underset{N}{\underline{\lim}} \; \mbox{\f $\dis\frac{1}{N^{d-2}}$} \;\log \ov{\IP} [\rho_{N,u} \in \cD_{a,\delta}] \ge -  \mbox{\f $\dis\frac{1}{d}$} \; \big(\sqrt{a} - \sqrt{u}\big)^2 {\rm cap}  (K^\delta).
\end{equation}
\end{theorem}

\begin{proof}
We consider $\ve > 0$ and define
\begin{equation}\label{6.17}
\mu_\ve = \big(\sqrt{u} +\big(\sqrt{a + \ve} - \sqrt{u}\big)\,h\big)^2\,m_B,
\end{equation}

\n
where $h(y) = P_y[\wt{H}_{K^\delta} < \infty] = P_y[H_{K^\delta} < \infty]$, for $y \in \IR^d$, is the equilibrium potential of $K^\delta$ (every point of $K^\delta$ is regular for $K^\delta$ so $h \in C_0(\IR^d)$ and $h = 1$ on $K^\delta$). Observe that
\begin{equation}\label{6.18}
\begin{split}
r_\delta(\mu_\ve) (y)  \stackrel{(\ref{6.4})}{=} &\; \dis\int_B  \big(\sqrt{u} +\big(\sqrt{a + \ve} - \sqrt{u}\big) \,h(z)\big)^2 \varphi_\delta (y-z)dz
\\[1ex]
 \ge&\; a + \ve \; \mbox{, when} \; y \in K
\\
&\; \mbox{(since $h=1$ on $K^\delta$, and this is actually an equality)}.
\end{split}
\end{equation}

By (\ref{6.6}), we can find an open neighborhood $O$ of $\mu_\ve$ in $M_+(B)$ so that for all $\mu \in O$, $r_\delta(\mu) \ge a + \frac{\ve}{2}$ on $K$. It follows that $\{r_\delta(\mu) \ge a\}$ separates $K$ from $\partial B_0$, for all $\mu \in O$, i.e.~$O \subseteq \cD_{a,\delta}$. As a result of Theorem \ref{theo5.8} we find that
\begin{equation}\label{6.19}
\underset{N}{\underline{\lim}} \; \mbox{\f $\dis\frac{1}{N^{d-2}}$} \;\log \ov{\IP} [\rho_{N,u} \in \cD_{a,\delta}] \ge  -\inf\limits_O \;  \mbox{\f $\dis\frac{1}{d}$} \; I_u \ge  - \mbox{\f $\dis\frac{1}{d}$} \; I_u(\mu_\ve).
\end{equation}
Now $h$ belongs to $\cF_e$ and $\cE(h,h) = {\rm cap}(K^\delta)$, see \cite{FukuOshiTake94}, p.~71. Therefore, by (\ref{3.4}), (\ref{3.5}) and below (\ref{2.27}),
\begin{equation}\label{6.20}
I_u(\mu_\ve) = \big(\sqrt{a + \ve} - \sqrt{u}\big)^2 \,\cE(h,h) = \big(\sqrt{a + \ve} - \sqrt{u}\big)^2 \,{\rm cap}(K^\delta).
\end{equation}

\n
Inserting this identity in (\ref{6.19}) and letting $\ve \r 0$, we obtain (\ref{6.16}).
\end{proof}

\begin{remark}\label{rem6.5} \rm ~

\medskip\n
1) One knows from Proposition 1.13, p.~60 of  \cite{PortSton78}, that  ${\rm cap}(K^\delta) \downarrow {\rm cap}(K)$, as $\delta \r 0$. So, when $\delta \r 0$, the constant in the right-hand side of the lower bound (\ref{6.16}) tends to the constant in the right-hand side of the upper bound (\ref{6.9}).

\medskip\n
2) Actually, one can let $\delta$ slowly tend to $0$ in (\ref{6.9}) and (\ref{6.16}). More precisely, given a choice of $\varphi_\delta$ for each $\delta \in (0,1)$, $a > u > 0$ and $K$, $B_0$, $B$ satisfying (\ref{6.3}) when $\delta = \frac{1}{2}$, one can, using a diagonal type procedure, the remark above, and Theorems \ref{theo6.2}, \ref{theo6.4}, to construct a sequence $\delta_N$ slowly tending to zero so that
\begin{equation}\label{6.21}
\lim\limits_N \; \mbox{\f $\dis\frac{1}{N^{d-2}}$}  \;\log \ov{\IP} [\rho_{N,u} \in \cD_{a,\delta_N}] = -  \mbox{\f $\dis\frac{1}{d}$} \big(\sqrt{a} - \sqrt{u}\big)^2 \;{\rm cap}(K).
\end{equation}

\medskip\n
3) We can, in place of $\cD_{a,\delta}$ in (\ref{6.5}), instead consider ``disconnection by sub-level sets'', i.e.
\begin{equation}\label{6.22}
\mbox{$\cD'_{a,\delta} = \{\mu \in M_+(B); \;\{r_\delta(\mu) \le a\}$ disconnects $K$ from $\partial B_0\}$}
\end{equation}

\medskip\n
(replacing $\sup_{0 \le t \le 1} f(\psi(t)) \ge a$ by $\inf_{0 \le t \le 1} f(\psi(t)) \le a$ in (\ref{6.2}) to define the above event).

\medskip
The same arguments of Theorems \ref{theo6.2} and \ref{theo6.4} apply (one replaces $a + \ve$ by $a - \ve$, in (\ref{6.17}), with $0 < \ve < a$) and we obtain that for $0 < a < u$, $0 < \delta < 1$, under (\ref{6.3})
\begin{equation}\label{6.23}
\begin{split}
\overline{\lim\limits_N}  \;\mbox{\f $\dis\frac{1}{N^{d-2}}$} \;\log \ov{\IP} \big[  \rho_{N,u} \in \cD'_{a,\delta}\big]  & \le - \mbox{\f $\dis\frac{1}{d}$} \big(\sqrt{a} - \sqrt{u}\big)^2 \;{\rm cap}(K),
\\[2ex]
\overline{\lim\limits_N}  \;\mbox{\f $\dis\frac{1}{N^{d-2}}$} \;\log \ov{\IP} \big[  \rho_{N,u} \in \cD'_{a,\delta}\big]  & \ge - \mbox{\f $\dis\frac{1}{d}$} \big(\sqrt{a} - \sqrt{u}\big)^2 \;{\rm cap}(K^\delta), 
\end{split}
\end{equation}

\medskip\n
4) It is instructive to compare (\ref{6.16}) in Theorem \ref{theo6.4} with the lower bound one obtains by the following intuitive ``change of measure'' strategy. Namely, assume $0 < u < a$ and introduce the new probability measure 
\begin{equation*}
\wt{\IP}_N = e^{\lambda \eta - u \,{\rm cap}_{\IZ^d}(K^{\delta,N})(e^\lambda - 1)} \,\ov{\IP},
\end{equation*}

\n
where $\eta$ stands for the total number of bilateral trajectories modulo time-shift with label at most $u$, which enter $K^{\delta,N}= (NK^\delta) \cap \IZ^d$, $\lambda = \log (\frac{a + \ve}{u})$, and we recall that for $A$ finite in $\IZ^d$, ${\rm cap}_{\IZ^d}(A)$ stands for the capacity of $A$ (attached to the simple random walk on $\IZ^d$).

\medskip
Under $\wt{\IP}_N$, the variable $\eta$ has Poisson distribution with parameter $(a + \ve) \,{\rm cap}_{\IZ^d}(K^{\delta,N})$ (instead of $u\,{\rm cap}_{\IZ^d}(K^{\delta,N})$ under $\ov{\IP}$), and $(L_{x,u})_{x \in K^{\delta,N}}$ has the same distribution as \linebreak $(L_{x,a +\ve})_{x \in K^{\delta,N}}$ under $\ov{\IP}$. As we now explain,
\begin{equation}\label{6.24}
\lim\limits_N \;\wt{\IP}_N(\rho_{N,u} \in \cD_{a,\delta}) = 1.
\end{equation}

\medskip\n
Indeed, by (\ref{4.4}), $\ov{\IP}$-a.s., $\frac{1}{N^d} \,\sum_{y \in \IL_N} \, L_{N y,a+\ve}\,\delta_y$ converges vaguely to $(a + \ve)dy$, as $N \r \infty$, and the restriction to $K$ of $r_\delta(\rho_{N,u})$ under $\wt{\IP}_N$ has the same distribution as the restriction to $K$ of $r_\delta(\rho_{N,a + \ve})$ under $\ov{\IP}$, which $\ov{\IP}$-a.s. converges uniformly on $K$ to $a + \ve$. However, as soon as the restriction to $K$ of $r_\delta(\rho_{N,u})$ exceeds $a$ everywhere, the event $\{\rho_{N,u} \in \cD_{a,\delta}\}$ occurs, and (\ref{6.24}) follows.

\medskip
By the classical relative entropy estimate, see for instance \cite{DeusStro89}, p.~76, and (\ref{6.24}),
\begin{equation}\label{6.25}
\liminf\limits_N \; \mbox{\f $\dis\frac{1}{N^{d-2}}$} \,\log \ov{\IP} [\rho_{N,u} \in \cD_{a,\delta}] \ge - \limsup\limits_N \; \mbox{\f $\dis\frac{1}{N^{d-2}}$} H(\wt{\IP}_N\, | \,\ov{\IP}),\end{equation}

\n
where $H(\wt{\IP}_N \,|\,\ov{\IP}) = \IE^{\wt{\IP}_N } [\log  \mbox{\f $\dis\frac{d\wt{\IP}_N}{d\ov{\IP}}$}]$ stands for the relative entropy of $\wt{\IP}_N$ with respect to $\ov{\IP}$. As we now explain,
\begin{equation}\label{6.26}
\liminf\limits_N \; \mbox{\f $\dis\frac{1}{N^{d-2}}$} H(\wt{\IP}_N\, | \,\ov{\IP}) \ge \Big(a \log \mbox{\f $\dis\frac{a}{u}$} - a + u\Big) \;\mbox{\f $\dis\frac{1}{d}$}\;{\rm cap}(K^\delta),
\end{equation}
and since (as can be checked directly)
\begin{equation}\label{6.27}
v \log  \mbox{\f $\dis\frac{v}{u}$} - v + u > \big(\sqrt{v} - \sqrt{u}\big)^2, \; \mbox{for any $v,u > 0$, with $v \not= u$},
\end{equation}
the lower bound (\ref{6.25}) is worse than (\ref{6.16}).

\medskip
To prove (\ref{6.26}), one notes that by our choice of $\wt{\IP}_N$,
\begin{equation}\label{6.28}
\begin{split}
H(\wt{\IP}_N\,|\,\ov{\IP}) & = E^{\wt{\IP}_N}[\lambda \eta] - (a + \ve - u) \;{\rm cap}_{\IZ^d}(K^{\delta,N})
\\[1ex]
& = \big((a + \ve) \;\log  \big(  \mbox{\f $\dis\frac{a+ \ve}{u}$} \big) - (a + \ve) + u\big)  \;{\rm cap}_{\IZ^d}(K^{\delta,N}).
\end{split}
\end{equation}

\n
The same argument leading to (\ref{5.37}), (\ref{5.42}) shows that for each $R \ge 1$ and $\delta' \in (0, \delta)$ one can construct $\Phi \in H^1(\IR^d)$ with compact support, which is a.e.\,equal to $1$ on $K^{\delta '}$, and such that $\cE(\Phi,\Phi) \le d(1 + \frac{c(B)}{R^{d-2}}) \,\beta$ with $\beta = \underline{\lim}_N \; \frac{1}{N^{d-2}} \;{\rm cap}_{\IZ^d}(K^{\delta,N})$. It thus follows (see~\cite{FukuOshiTake94}, p.~71) that
\begin{equation*}
\beta \ge  \mbox{\f $\dis\frac{1}{d}$}  \; \Big (1 + \mbox{\f $\dis\frac{c(B)}{R^{d-2}}$}\Big)^{-1}\;{\rm cap}(K^{\delta '}).
\end{equation*}

\medskip\n
Letting $R \r \infty$, $\delta ' \uparrow \delta$, and noting that when $y \in K^\delta$, $P_y$-a.s. Brownian motion immediately hits $\bigcup_{\delta' < \delta} K^{\delta '}$, so that by Theorem 1.10 and Proposition 1.13, p.~58 and 60 of \cite{PortSton78}, ${\rm cap}(K^{\delta '}) \uparrow {\rm cap} (K^\delta)$ as $\delta ' \uparrow \delta$, we obtain that
\begin{equation}\label{6.29}
\liminf\limits_N \; \mbox{\f $\dis\frac{1}{N^{d-2}}$}  \;{\rm cap}_{\IZ^d}(K^{\delta,N}) \ge \mbox{\f $\dis\frac{1}{d}$}  \;{\rm cap}(K^\delta).
\end{equation}

\n
Since $v \r v \log \frac{v}{u} - v + u$ is increasing for $v \in (u,\infty)$, the combination of (\ref{6.28}) and (\ref{6.29}) readily yields (\ref{6.26}).

\medskip
So, the intuitive lower bound we just described does not capture (\ref{6.16}).

\bigskip\n
5) It is an important feature of random interlacements that the vacant set $\cV^u$ of random interlacements at level  $u$ on  $\IZ^d, d \ge 3$, undergoes a phase transition from a percolative regime, when $u < u_*$, to a non-percolative regime, when $u > u_*$, with $u_*$ a certain non-degenerate critical value, which is positive  and finite (see \cite{Szni10a}, \cite{SidoSzni09}, and also \cite{DrewRathSapo12}, \cite{PopoTeix13} for recent developments).  Given a smooth compact subset $K$ of $\IR^d$ and its discrete blow-up $K^N = (NK) \cap \IZ^d$, one can consider the disconnection event $\{K^N \overset{\cV^u}{\longleftrightarrow}\hspace{-3ex}\mbox{\scriptsize  /}\quad  \infty\}$, for which the connected components of $K^N \cap \cV^u$ in $\cV^u$ are finite (possibly empty). Looking at a small interior ball in $K$ and its discrete blow-up, it is straightforward to argue that for $u < u_*$, $\lim_N \ov{\IP} [K^N  \overset{\cV^u}{\longleftrightarrow}\hspace{-3ex}\mbox{\scriptsize  /}\quad  \infty] = 0$. One can wonder whether the main effect in realizing this atypical disconnection event stems from a large deviation of the density profile of occupation-times, for which, roughly speaking, values of the profile exceeding $u_*$ would insulate $K$ from infinity, and whether one has the asymptotics
\begin{equation*}
\lim\limits_N \; \mbox{\f $\dis\frac{1}{N^{d-2}}$} \;\log \ov{\IP} [K^N   \overset{\cV^u}{\longleftrightarrow}\hspace{-3ex}\mbox{\scriptsize  /}\quad  \infty] = - \mbox{\f $\dis\frac{1}{d}$} \;\big(\sqrt{u}_* - \sqrt{u})^2  \,{\rm cap}(K), \;\mbox{for} \; u < u_*.
\end{equation*}

\medskip\n
We refer to \cite{LiSzni13b}, for a lower bound on the left-hand side of a similar nature.

 \hfill $\square$
\end{remark}

\end{document}